\documentclass[11pt]{amsart}

\usepackage[margin=1in]{geometry}
\usepackage{amsmath,amssymb,amsthm,mathtools}
\usepackage{enumitem}

\ifx\pdfpageheight\undefined
   \usepackage[dvips,colorlinks=true,linkcolor=blue,citecolor=red,%
      urlcolor=green]{hyperref}
   \usepackage[dvips]{graphicx}
   \makeatletter
   \edef\Gin@extensions{\Gin@extensions,.mps}
   \makeatother
\else
 \usepackage[pdftex]{graphicx}
   \usepackage[bookmarksopen=false,pdftex=true,breaklinks=true,%
      backref=page,pagebackref=true,plainpages=false,%
      hyperindex=true,pdfstartview=FitH,colorlinks=true,%
      pdfpagelabels=true,colorlinks=true,linkcolor=blue,%
      citecolor=red,urlcolor=green,hypertexnames=false%
      ]%
   {hyperref}
\fi

\theoremstyle{theorem}
\newtheorem{theorem}{Theorem}
\newtheorem{lemma}{Lemma}[section]
\newtheorem{proposition}{Proposition}[section]

\newtheorem{conjecture}{Conjecture}[section]

\theoremstyle{remark}
\newtheorem{remark}{Remark}[section]

\theoremstyle{definition}
\newtheorem{definition}{Definition}[section]
\newtheorem{notation}{Notation}[section]
\newtheorem{example}{Example}[section]

\newcommand{\R}{\mathbb{R}}
\newcommand{\im}{\operatorname{im}}

\newcommand{\Harm}{\mathcal{H}}
\newcommand{\argmin}{\operatorname*{arg\,min}}

\usepackage{amssymb,color,epsf}
\usepackage{mathtools}
\usepackage{enumitem}

\usepackage{bm}
\usepackage{amsmath}
\usepackage{tabularx}
\usepackage{color, colortbl}
\usepackage{url}
\usepackage{enumitem}

\DeclareMathAlphabet{\mathpzc}{OT1}{pzc}{m}{it}

\usepackage [cmtip,arrow]{xy}
\xyoption{all}
\usepackage {pb-diagram,pb-xy}
\usepackage[mathscr]{eucal}
\usepackage[utf8]{inputenc}
\usepackage[T1]{fontenc}
\usepackage[english]{babel}
\usepackage{bm}
\usepackage{tabularx}

\usepackage{bm}

\usepackage{blkarray, bigstrut}
\usepackage{gauss}

\newcommand{\sqbracket}[1]{[#1]}

\newcommand {\junk}[1]{}

\newcommand {\Z}  {\mathbb{Z}}

\newcommand {\la}   {{\langle}}
\newcommand {\ra}   {{\rangle}}

\newcommand {\Ker}      {\mathrm{Ker}}

\newcommand {\PP}     {\mathbb{P}} 

\newcommand{\card}{\mathrm{card}}

\def\addots{\mathinner{\mkern1mu
\raise1pt\vbox{\kern7pt\hbox{.}}
\mkern2mu\raise4pt\hbox{.}\mkern2mu
\raise7pt\hbox{.}\mkern1mu}}

\newcommand{\HH}  {\mbox{\rm H}}

\newcommand{\hide}[1]{}

\newcommand{\supp}{\mathrm{supp}}

\newcommand{\nc}{\newcommand}
\newcommand{\rc}{\renewcommand}
\nc{\mc}{\mathcal}
\rc{\t}{\text}
\nc{\op}[1]{\operatorname{#1}}
\nc{\opcat}[1]{\mathbf{#1}}

\nc{\id}{\op{id}}

\nc{\umutnote}[1]{{\marginpar{\small \textcolor{blue}{#1}}}}

\nc{\cA}{\mc{A}}\nc{\cB}{\mc{B}}\nc{\cC}{\mc{C}}\nc{\cD}{\mc{D}}\nc{\cE}{\mc{E}}\nc{\cF}{\mc{F}}\nc{\cG}{\mc{G}}\nc{\cH}{\mc{H}}\nc{\cI}{\mc{I}}\nc{\cJ}{\mc{J}}\nc{\cK}{\mc{K}}\nc{\cL}{\mc{L}}\nc{\cM}{\mc{M}}\nc{\cN}{\mc{N}}\nc{\cO}{\mc{O}}\nc{\cP}{\mc{P}}\nc{\cQ}{\mc{Q}}\nc{\cR}{\mc{R}}\nc{\cS}{\mc{S}}\nc{\cT}{\mc{T}}\nc{\cU}{\mc{U}}\nc{\cV}{\mc{V}}\nc{\cW}{\mc{W}}\nc{\cX}{\mc{X}}\nc{\cY}{\mc{Y}}\nc{\cZ}{\mc{Z}}

\rc{\PP}{\mathbb{P}}
\rc{\AA}{\mathbb{A}}
\nc{\bbC}{\mathbb{C}}
\nc{\CC}{\mathbb{C}}

\nc{\code}[1]{{\texttt{#1}}}
\nc{\mcode}[1]{{\text{\texttt{#1}}}}
\nc{\xto}[1]{\raisebox{-0.03cm}{\scalebox{0.85}{$\,\xrightarrow{#1}\,$}}}
\nc{\xtonormal}[1]{\xrightarrow{#1}}
\nc{\xfrom}[1]{\xleftarrow{#1}}
\nc{\sidenote}[1]{\marginpar{\small #1}}

\nc{\Aff}{\opcat{Aff}}
\nc{\AffVar}{\opcat{AffVar}}
\nc{\ProjVar}{\opcat{ProjVar}}
\nc{\GAP}{\opcat{GrAlgPairs}}
\nc{\GA}{\opcat{GrAlg}}

\nc{\acc}{\mathrm{a.c.c}}
\nc{\GL}{\mathrm{GL}}

\nc{\Mod}{\t{-}\opcat{Mod}}
\nc{\Sub}{\opcat{Sub}}
\nc{\iso}{\cong}
\nc{\compose}{\circ}

\newcommand{\FF}{\mathcal{F}}

\newcommand{\spn}{\mathrm{span}}

\newcommand{\bp}{\begin{proof}}
\newcommand{\ep}{\end{proof}}

\newcommand{\Rep}{\mathrm{Rep}}
\newcommand{\econt}{\mathrm{content}}

\begin{document}
\title[Essential Simplices Dominate in Harmonic Representatives]{Essential Simplices Dominate in Harmonic Representatives of One-Dimensional Persistent Classes}

\author{Saugata Basu}
\address{Department of Mathematics,
Purdue University, West Lafayette, IN 47906, U.S.A.}
\email{sbasu@math.purdue.edu}

\author{Aldo Guzmán-Sáenz}
\address{Independent Researcher}
\email{aldo.guzman.saenz@gmail.com}

\author{Laxmi Parida}
\address{IBM Research, Yorktown Heights}
\email{parida@us.ibm.com}

\date{}
\keywords{Harmonic persistent homology, essential simplices, harmonic weights, graph Laplacian}
\begin{abstract}
 Persistent homology summarizes 
   the birth and death of topological features, but it does not by itself specify where a feature is located in the underlying complex. Harmonic persistent homology addresses this by assigning canonical harmonic cycle representatives to bars. In earlier work, Basu and Cox \cite{Basu-Cox} showed that harmonic representatives of simple bars maximize the total relative weight placed on essential simplices, the simplices that are forced to appear in representatives of the corresponding class.
In this paper we prove that, for generic one-dimensional bars, this preference is stronger than an aggregate maximization statement. Every essential edge has strictly larger coefficient, in absolute value, than every non-essential edge in the harmonic representative, and the absolute values of the coefficients of all essential edges are equal.

The key argument is a finite-dimensional variational characterization of the harmonic representative as a minimum-norm chain with prescribed boundary, combined with an elementary graph-theoretic cut argument.
We then prove that the result is special to dimension one. In higher dimensions, the analogous coefficient-wise dominance statement fails.

We give examples to show that harmonic representatives can place larger coefficients on non-essential higher-dimensional simplices than on essential ones. These results clarify both the power and the limitations of using harmonic representatives to assign geometric significance to simplices in persistent homology. 
\end{abstract}

\maketitle

\tableofcontents
\section{Introduction}
\subsection{Persistent homology and barcodes}
The theory of persistent homology (see for example \cite{Edelsbrunner-Harer2010}) 
associates to any filtration of finite simplicial
complexes an object known as ``barcode'' of the filtration. 
The bars in a barcode correspond
loosely speaking to the lifetime of homology classes appearing in the homology of the 
simplicial complexes that appear in the filtration (here we are thinking of the ordered index set of the filtration
as time).
Precise definition (relevant to the current paper) is given below in Section~
\ref{subsec:precise}. To get an intuitive feeling for the definition of barcodes it may be instructive to study 
Example~\ref{eg:essential} and its accompanying
Figure~\ref{fig:essential}. 
Note however the important point that a new homology class that is ``born'' at a certain time  is defined only modulo a certain subspace in the homology of the complex at that time  -- thus identifying a bar with one particular homology class is problematic -- and no canonical identification exists.

\subsection{Associating cycles to bars}
Often in practice there is a demand to associate not just a homology class, but a specific 
\emph{cycle from the chain group representing this class} or at least a \emph{set of simplices} to each bar. 
This is because in applications
the simplices of the simplicial complexes of a filtration themselves  often have special significance. For instance, the vertices of a given simplicial complex could be labelled by
genes and a $p$-simplex $\sigma = (g_0,\ldots,g_p)$ may signify positive correlation between the genes 
$g_0,\ldots,g_p$ (say in causing a certain disease) (see for example
\cite{lockwood2014,Gurnari-et-al}).

There have been several approaches to the problem of
associating specific cycle representatives to persistent homology classes.
Most of these approaches involve minimization of some weight on the space of cycles representing a homology class.
For instance,
volume-optimal cycles were proposed in the non-persistent setting in \cite{Dey-2010} and in the setting of persistent homology in \cite{volume}. 

\subsection{Harmonic representatives of bars}
In \cite{Basu-Cox} a new
approach based on the theory of \emph{harmonic chains} was initiated. 
In that paper the authors consider homology groups with coefficients in $\R$ and impose an inner product on the chain
group to make the chain groups an Euclidean space.
As a result they are able to identify the various persistent 
homology groups, as well as the bars in the barcode of the given filtration, as subspaces of the  
simplicial chain groups themselves. 
Note that in contrast with ordinary persistent homology theory, one can associate canonically (only depending on the chosen inner product) a certain subspace
of the chain space to each bar. When the bar is of multiplicity one
this subspace is spanned by a single vector, and we have a uniquely defined (up to scalar multiplication)
cycle representing the bar --  such a cycle is called a \emph{harmonic representative} of the bar. 
The theory of \emph{harmonic persistence} has since found applications (see \cite{Gurnari-et-al}).
It has also led to a flurry of very interesting recent works \cite{hou2024,gulen2025grassmannianpersistencediagramsspecial} extending and reformulating the results in \cite{Basu-Cox} in several different directions.

\subsection{Essential simplices}
There are several reasons to consider harmonic representatives.
Instead of trying to optimize the length of the cycle the harmonic representatives 
put more relative weight on certain important simplices.
Since as remarked earlier, the simplices themselves in the simplicial complex underlying the filtration often have
domain dependent meaning --  if a particular simplex shows up with non-zero coefficient in \emph{every} cycle representing the homology class, then this fact may be considered
significant from an application point of view
(the lengths of representative cycles are not so significant
in these applications).
This idea was formalized in \cite{essential} where the notion
of \emph{essential} simplices corresponding to the bars of a barcode was introduced. Essential simplices (in dimension one) 
play a central role the application of topological data analysis to the study of population genomics. 
The notion of essential simplices is also central to the current paper. 
Informally, a simplex is essential relative to a bar, if it occurs with a non-zero coefficient in \emph{every}  cycle representing the bar. We will give 
a precise definition later taken from \cite{Basu-Cox}.

\subsection{Essential content}
\label{subsec:essential}
An important connection between the harmonicity of a representative
cycle and the set of essential simplices of a bar was established in \cite{Basu-Cox}.
If a bar in the barcode of a filtration
is of multiplicity one (this happens generically), then it is represented by a unique harmonic 
representative (unique up to multiplication by non-zero scalar). 
We define for each cycle,  
\[
z = \sum_{\sigma} c_\sigma \cdot \sigma,
\]
(not necessarily harmonic) representing any given simple bar the relative
essential content,
\[
\econt(z) = \left(\frac{\sum_{\sigma \mbox{ is essential}}  c_\sigma^2}{\sum_{\sigma} c_\sigma^2}\right)^{1/2},
\]
of the cycle 
which measures the relative weight in the cycle of the essential as opposed to the 
non-essential simplices. 

A key result proved in \cite{Basu-Cox}
is that the harmonic 
representatives of bars maximize (amongst all  representative cycles)
the relative essential content of the bar, 
i.e.
if $z_0$ is a harmonic representative of a simple bar $b$,
then for any cycle $z$  representing $b$,
\begin{equation}
\label{eqn:essential}
\econt(z) \leq \econt(z_0).
\end{equation}

Inequality~\eqref{eqn:essential} picks out the harmonic cycles as preferred
representatives of homology classes as they maximize the contributions of 
the essential simplices. 
We shift our focus in the current paper to the essential simplices themselves.

\subsection{Harmonic weights}
The numbers $|c_\sigma|$ (i.e. the absolute value of the coefficients of the 
simplex $\sigma$ in the harmonic representative of a bar) are called 
\emph{harmonic weights} in \cite{Gurnari-et-al}. 
Harmonic representatives maximize aggregate essential content \cite{Basu-Cox}. One might hope that the following stronger pointwise dominance property
\eqref{eqn:main} also holds.
Given a harmonic representative $h = \sum_{\sigma} c_\sigma \cdot \sigma$ of a simple bar $b$,
and any two simplices $\sigma_0$ and $\tau_0$ of the same dimension 
with $\tau_0 \neq \sigma_0$, and 
such that $\sigma_0$ is essential for the bar $b$,

\begin{equation}
\label{eqn:main}
|c_{\tau_0}| \leq |c_{\sigma_0}|
\end{equation}
with equality if and only if $\tau_0$ is also essential for the bar $b$.

This statement is made precise later (see Conjecture~\ref{conj:essential} below) after we have introduced the relevant definitions. The main purpose of the current paper is to investigate the conditions under which  inequality~\eqref{eqn:main} is valid. 
Surprisingly, as we show in this paper,
it is true only in dimension one and fails fundamentally in higher dimensions.

Note however that in many applications one-dimensional bars play the most important role. 
For example, they are used  to generate hypotheses in cancer genomics in \cite{Gurnari-et-al}, where
the weights in dimension one  played the most important role.
Indeed, essential simplices were introduced in \cite{essential} with the goal of applying topological data analysis techniques
to population genomics. Again bars of dimension one were the ones of most significance in this application. This makes the fact that inequality~\eqref{eqn:main} holds in dimension one
significant from the applications point of view.


\subsection{Precise definitions}
\label{subsec:precise}
We assume that the reader is already familiar with basic definitions of simplicial homology.  We fix below some notation.

\begin{definition}
\label{def:simplicial-complex}
A \emph{finite simplicial complex} $K$ is a set of ordered subsets of $[N] = \{0,\ldots N\}$ for some $N \geq 0$, such that if $\sigma \in K$ and $\tau$
is a subset of $\sigma$, then $\tau \in K$. 
\end{definition}

\begin{notation}
\label{def:p-simplices}
If $\sigma = \{i_0,\ldots,i_p\} \in K$, 
with $K$ a finite simplicial complex,
and $i_0 < \cdots < i_p $, we will denote $\sigma = [i_0,\ldots,i_p]$ and call $\sigma$ a \emph{$p$-dimensional simplex of $K$}. 
We will denote by $K^{(p)}$ the subcomplex of $K$ consisting of simplices of $K$ of dimension $\leq p$.
We will denote by $K^{\sqbracket{p}} = K^{(p)} \setminus K^{(p-1)}$ the subset of $p$-dimensional simplices of $K$.
Note that $K^{\sqbracket{p}}$ is not a subcomplex.
\end{notation}

\begin{definition}[Chain groups]
\label{def:chain-groups}
Suppose $K$ is a finite simplicial complex.
For $p \geq 0$,  we will denote by 
$C_p(K) = C_p(K;\R)$ (the $p$-th chain group),   
the $\R$-vector space generated by the elements of $K^{\sqbracket{p}}$,
i.e.
\[
C_p(K) = \bigoplus_{\sigma \in K^{\sqbracket{p}}}\R \cdot \sigma.
\]
\end{definition}

\begin{definition}[The boundary map]
\label{def:boundary-map}
We denote by $\partial_p(K): C_p(K) \rightarrow C_{p-1}(K)$ the 
linear map (called the $p$-th \emph{boundary map})  defined as follows. Since 
$\left(\sigma \right)_{\sigma \in K^{\sqbracket{p}}}$ is a basis of $C_p(K)$ it is enough to define the image of each $\sigma \in C_p(K)$. 
We define for $\sigma = [i_0,\ldots,i_p] \in K^{\sqbracket{p}}$,
$
\partial_p(K)(\sigma) = \sum_{0\leq j \leq p} (-)^j [i_0,\ldots, \widehat{i_j}, \ldots,i_p] \in C_{p-1}(K),
$
where $\widehat{\cdot}$ denotes omission. (If the value of $p$ is clear from context we will sometimes drop the subscript and denote $\partial_p$ by $\partial$.)
\end{definition}

\begin{notation}[Cycles, boundaries, homology and the canonical surjection]
We denote 
$
Z_p(K) = \Ker(\partial_p(K)),
$
(the space of \emph{$p$-dimensional cycles}),
$
B_p(K) = \mathrm{Im}(\partial_{p+1}(K))
$
(the space of \emph{$p$-dimensional boundaries}),
and
$\HH_p(K) = Z_p(K)/B_p(K)$
(the \emph{$p$-dimensional simplicial homology group} of $K$).
We will denote by 
\[
\phi_p(K):Z_p(K) \rightarrow Z_p(K)/B_p(K) =  \HH_p(K)
\]
the canonical surjection. 
(Unless stated otherwise all homology groups in this paper are taken with real coefficients).
\end{notation}

\subsection{Representing homology classes by harmonic chains} 
Let $K$ be a finite simplicial complex.
We make the chain group $C_p(K)$ into an Euclidean space by fixing an inner product $\la \cdot,\cdot\ra_{C_p(K)}$. 
For the rest of the paper we fix the following inner product on $C_p(K)$ which we will    
refer to as the standard inner product on $C_p(K)$.
We define:
\begin{equation}
\label{eqn:standard}
\la \sigma, \sigma'\ra_{C_p(K)} = \delta_{\sigma,\sigma'},  \sigma,   \sigma' \in K^{\sqbracket{p}}
\end{equation}
(i.e. we declare the basis  $\left(\sigma \right)_{\sigma \in K^{\sqbracket{p}}}$ to be an orthonormal basis).
If the context is clear we will omit the subscript from the notation  $\la \cdot,\cdot\ra_{C_p(K)}$.

We now come to a key definition -- namely, that of harmonic homology (following \cite{Eckmann}).
\begin{definition}[Harmonic homology subspace]
\label{def:harmonic}
For $p \geq 0$,  we will denote 
\[
\mathcal{H}_p(K) = Z_p(K)  \cap  B_p(K)^\perp
\]
and call $\mathcal{H}_p(K) \subset C_p(K)$ the \emph{harmonic homology subspace of $K$}. 
\end{definition}

\subsubsection{Elementary properties}
The following  propositions encapsulate the key properties of the harmonic homology subspaces. We denote by $\mathrm{proj}_W:V \rightarrow V$ the orthogonal projection on to a subspace $W$ of a vector space $V$.
\begin{proposition}
\label{prop:f}
The map  $\mathfrak{f}_p(K)$ defined by 
\begin{equation}
    \label{eqn:f}
    z + B_p(K) \rightarrow \mathrm{proj}_{B_p(K)^\perp}(z), z \in Z_p(K)
\end{equation}
gives an isomorphism
$
\mathfrak{f}_p(K): \HH_p(K)  \rightarrow \mathcal{H}_p(K).
$
\end{proposition}

\begin{proof}
See \cite{Basu-Cox}.
\end{proof}

\begin{proposition}
\label{prop:proj}
For all $z \in Z_p(K)$,
    \[
    \mathfrak{f}_p(K) \circ \phi_p(K)(z)  = \mathrm{proj}_{\mathcal{H}_p(K)}(z).
    \]
\end{proposition}

\begin{proof}
    Follows immediately from Proposition~\ref{prop:f}. 
\end{proof}

\begin{proposition}
\label{prop:intersection}
$
\mathcal{H}_p(K) = \Ker(\partial_{p+1}(K)^*) \cap \Ker(\partial_{p}(K))
$
(where $L^*$ denotes the adjoint of a linear map $L$ between two inner product spaces).
\end{proposition}

\begin{proof}
See \cite{Basu-Cox}.
\end{proof}

\begin{remark}
\label{rem:harmonic}
The harmonic homology group
$\mathcal{H}_p(K)$ as defined above is
equal to the kernel of the linear map 
$\Delta_p = \partial_{p+1}\circ \partial_{p+1}^* + \partial_p^*\circ \partial_p$.
The linear map $\Delta_p(K):C_p(K) \rightarrow C_p(K)$ is a discrete analog of the Laplace operator and 
thus it makes sense to call its kernel the space of harmonic cycles.
\end{remark}

\subsubsection{Functoriality of the maps $\mathfrak{f}_p(K)$ under inclusion}
Now suppose $K_1 \subset K_2$ are sub-complexes of the finite simplicial complex $K$. 
Then, $C_p(K_1)$ is a subspace of $C_p(K_2)$. 

\begin{proposition}
\label{prop:functorial}
The restriction of $\mathrm{proj}_{B_p(K_2)^\perp}$ to $\mathcal{H}_p(K_1)$ gives a linear map
\[
\mathfrak{i}_p= \mathrm{proj}_{B_p(K_2)^\perp}|_{\mathcal{H}_p(K_1)} : \mathcal{H}_p(K_1) \rightarrow \mathcal{H}_p(K_2),
\]
which makes the following diagram commute
\[
\xymatrix{
\HH_p(K_1) \ar[r]^{i_p} \ar[d]^{\mathfrak{f}_p(K_1)} & \HH_p(K_2) \ar[d]^{\mathfrak{f}_p(K_2)} \\
\mathcal{H}_p(K_1) \ar[r]^{\mathfrak{i}_p} &\mathcal{H}_p(K_2)
}
\]
where $i_p: \HH_p(K_1) \rightarrow \HH_p(K_2)$ is the map induced by the inclusion $K_1 \hookrightarrow K_2$.
\end{proposition}

\begin{proof}
    See \cite{Basu-Cox}.
\end{proof}

\subsection{Persistent homology and barcodes}
\label{subsec:persistent}
There are many equivalent ways of defining barcodes of filtrations
(see for example the book \cite{Edelsbrunner-Harer2010}).
We use in this paper the definitions introduced in \cite{Basu-Karisani}, since they allow us to define harmonic analogs of barcodes in a seamless manner.

Let $N > 0$, and $T$ denote the ordered set $[0,N]$. 
Let
$\mathcal{F} = (K_t)_{t \in T}$,  be a tuple of sub-complexes of a  finite simplicial complex $K$, such that
$s \leq t \Rightarrow K_s \subset K_t$.    We call $\mathcal{F}$ a filtration
of the  simplicial complex $K$.

\begin{notation}
\label{not:inclusion}

  For $s, t \in T,   s \leq t$, and $p \geq 0$, we let $i_p^{s, t} : \HH_p (K_s) \longrightarrow
  \HH_p (K_t)$, denote the homomorphism induced by the inclusion $K_s
  \hookrightarrow K_t$.
\end{notation}

\begin{definition}[Persistent homology groups]
\label{def:persistent}
 For each triple $(p, s, t)  \in \Z_{\geq 0} \times T  \times T$ with 
$s \leq t$ 
  the   
  \emph{persistent homology  group},       
  $\HH_p^{s, t} (\mathcal{F})$ is defined by
  \begin{eqnarray*}
    \HH_p^{s, t} (\mathcal{F}) & = &  \mathrm{Im} (i_p^{s, t}).
  \end{eqnarray*}
  Note that $\HH_p^{s, t} (\mathcal{F}) \subset \HH_p (K_t)$,    
  and 
  $\HH_p^{s, s}(\mathcal{F}) = \HH_p (K_s)$.
\end{definition}

\subsubsection{Barcodes of filtrations}
The following definitions are taken from \cite{Basu-Karisani} (see also the references therein and \cite[Theorem 1]{Ghrist-Henselman}).
We follow  the same notation as above and first define certain subspaces of the homology
groups $\HH_p(K_s), s \in T, p \geq 0$.

\begin{definition}
\label{def:barcode}
Suppose $T = [0,N] \subset \mathbb{N}$. Then for
$0 \leq s < t \leq N$, 
and $p \geq 0$,
\begin{eqnarray*}
M^{s,t}_p(\mathcal{F}) &=& (i^{s,t}_p)^{-1}(\HH^{s-1,t}_p(\mathcal{F})), 
\end{eqnarray*}
and
\begin{eqnarray*}
P^{s,t}_p(\mathcal{F}) &=& M^{s,t}_p(\mathcal{F})/M^{s,t-1}_p(\mathcal{F}), \\
P^{s,\infty}_p(\mathcal{F}) &=&  
\HH_p(K_s) / M^{s,N}_p(\mathcal{F}).
\end{eqnarray*}
\end{definition}

\begin{definition}[Persistent multiplicity, barcode, simple bars]
\label{def:barcode2}
We will denote for $s \in T,  t \in T \cup \{\infty \}$,
\begin{equation}
\label{eqn:def:barcode:multiplicity}
\mu^{s,t}_p(\mathcal{F}) = \dim P^{s,t}_p(\mathcal{F}),
\end{equation}
and call $\mu^{s,t}_p(\mathcal{F})$ the \emph{persistent multiplicity of $p$-dimensional cycles born at time $s$ and dying at time $t$ if $t \neq \infty$,   or never dying in case $t = \infty$}.

Finally, we will call the set
\[
\mathbf{B}_p(\mathcal{F}) = \{(s,t;\mu^{s,t}_p(\mathcal{F})) \mid \mu^{s,t}_p(\mathcal{F}) > 0\}
\]
\emph{the $p$-dimensional barcode associated to the filtration $\mathcal{F}$}. 
\end{definition}

\begin{definition}[Simple bars]
\label{def:simple}
We call
$b = (s,t;\mu^{s,t}_p(\mathcal{F})) \in \mathbf{B}_p(\mathcal{F})$ a 
\emph{bar of $\mathcal{F}$ of multiplicity $\mu^{s,t}_p(\mathcal{F})$}. 
If $\mu^{s,t}_p(\mathcal{F})=1$, we  call $b$ a \emph{simple} bar.
\end{definition}

We need the following more stringent property on the bars of a barcode introduced in \cite{Basu-Cox} 
 (that is generically satisfied in 
filtrations arising from taking sub-level complexes of admissible functions
on finite simplicial complexes).

\begin{definition}[Generic bars] \cite{Basu-Cox}
\label{def:generic}
 We will say that a bar ${b} = (s,t;\mu^{s,t}_p(\mathcal{F})) \in \mathbf{B}_p(\mathcal{F})$ is \emph{generic} if it satisfies the following two conditions.
\begin{enumerate}
    \item  \label{itemlabel:generic:1}
    $b$ is simple, i.e. $\mu^{s,t}_p(\mathcal{F}) =1$;
    \item 
    \label{itemlabel:generic:2}
    for every $t' \in T \cup \{\infty\},  t' > s, t' \neq t$, $\mu^{s,t'}_p(\mathcal{F}) = 0$ (so no other bar in $\mathbf{B}_p(\mathcal{F})$ has birth time $s$).
\end{enumerate}
\end{definition}

\begin{remark}
    All bars appearing in the barcode of a filtration induced by generic admissible maps on a simplicial complex $K$ will be generic \cite{Basu-Cox}. So this assumption is not very restrictive.
\end{remark}

\subsection{Harmonic persistent homology and harmonic barcodes}
\label{subsec:harmonic-peristent-homology}
We now define the harmonic versions of persistent homology groups. These will all be subspaces
of $C_p(K), p \geq 0$.

Using Proposition~\ref{prop:functorial} we have for each $s,t \in T,s \leq t$, a linear map
\[
\mathfrak{i}_p^{s,t} := \mathrm{proj}_{B_p(K_t)^\perp}|_{\mathcal{H}_p(K_s)} : \mathcal{H}_p(K_s) \rightarrow \mathcal{H}_p(K_t),
\]

which makes the following diagram commute
\begin{equation}
\label{eqn:commutative}
\xymatrix{
\HH_p(K_s) \ar[r]^{i_p^{s,t}} \ar[d]^{\mathfrak{f}_p(K_s)} & \HH_p(K_t) \ar[d]^{\mathfrak{f}_p(K_t)} \\
\mathcal{H}_p(K_s) \ar[r]^{\mathfrak{i}_p^{s,t}} &\mathcal{H}_p(K_t)
}.
\end{equation}

\begin{definition}[Harmonic persistent homology subspaces]
\label{def:harmonic-persistent}
For each triple $(p, s, t)  \in \Z_{\geq 0} \times T  \times T$ with 
$s \leq t$ 
  the   
  \emph{harmonic persistent homology  subspace},       
  $\mathcal{H}_p^{s, t} (\mathcal{F})$ is defined by
  \begin{eqnarray*}
    \mathcal{H}_p^{s, t} (\mathcal{F}) & = &  \mathrm{Im} (\mathfrak{i}_p^{s, t}(\mathcal{F})) \subset C_p(K).
  \end{eqnarray*}
\end{definition}

\subsubsection{Harmonic barcodes of  filtrations}
We now give the harmonic analogs of the above spaces. They are all subspaces of $C_p(K)$ (in fact, of 
the various harmonic homology spaces $\mathcal{H}_p(K_s), s \in T$).

\begin{definition}[Harmonic barcode of a filtration]
\label{def:harmonic-barcode}
For $s \leq t$, and $p \geq 0$, we define
for
$0 \leq s < t \leq N$, 
and $p \geq 0$,
\begin{eqnarray*}
\mathcal{M}^{s,t}_p(\mathcal{F}) &=& (\mathfrak{i}^{s,t}_p)^{-1}(\mathcal{H}^{s-1,t}_p(\mathcal{F})), 
\end{eqnarray*}
and 
\begin{eqnarray*}
\mathcal{P}^{s,t}_p(\mathcal{F}) &=& 
\mathcal{M}^{s,t}_p(\mathcal{F})\cap\mathcal{M}^{s,t-1}_p(\mathcal{F})^\perp,\\ 
\mathcal{P}^{s,\infty}_p(\mathcal{F}) &=&  
\mathcal{H}_p(K_s) \cap \mathcal{M}^{s,N}_p(\mathcal{F})^\perp.
\end{eqnarray*}
\end{definition}

The vector spaces defined in Definition~\ref{def:harmonic-barcode} are isomorphic to
the corresponding ones in Definitions~\ref{def:barcode}.
Following the same notation as in Definition~\ref{def:harmonic-barcode} we have the 
following proposition.

\begin{proposition}
\label{prop:harmonic-barcode}
The map $\mathfrak{f}_p(K_s)$ induces isomorphisms:
\begin{eqnarray*}
\label{eqn:prop:harmonic-barcode:1}
\mathcal{M}^{s,t}_p(\mathcal{F}) &\cong & {M}^{s,t}_p(\mathcal{F}), \\ 
\label{eqn:prop:harmonic-barcode:3}
\mathcal{P}^{s,t}_p(\mathcal{F}) &\cong& {P}^{s,t}_p(\mathcal{F}), \\
\label{eqn:prop:harmonic-barcode:4}
\mathcal{P}^{s,\infty}_p(\mathcal{F}) &\cong&  {P}^{s,\infty}_p(\mathcal{F}).
\end{eqnarray*}
\end{proposition}

In analogy with Definition~\ref{def:barcode2} we now define harmonic barcodes of filtrations. 
\begin{definition}[Harmonic barcodes]
\label{def:harmonic-barcode2}
We will call the set
\[
\mathbf{\mathcal{B}}_p(\mathcal{F}) = \{(s,t;\mathcal{P}^{s,t}_p(\mathcal{F})) \mid  \mathcal{P}^{s,t}_p(\mathcal{F}) \neq  0\}
\]
\emph{the $p$-dimensional harmonic barcode associated to the filtration $\mathcal{F}$}.
For ${b} = (s,t;\mu^{s,t}_p(\mathcal{F})) \in {B}_p(\mathcal{F})$, we will call the subspace
$\mathcal{P}^{s,t}_p(\mathcal{F}) \subset \mathcal{H}_p(K_s)$, the harmonic homology subspace associated to
$b$.
\end{definition}

\begin{definition}[Harmonic representative]
\label{def:harmonic-rep}
Given a simple bar $b = (s,t,1)$ we will call any non-zero cycle in $\mathcal{P}^{s,t}_p(\mathcal{F})$
\emph{a harmonic representative of $b$}.
\end{definition}

\subsection{Essential simplices
}
\label{subsec:essential}
We now make precise the notion of essential simplices referred to previously.

Let $K$ be a finite simplicial complex and 
$\mathcal{F} = (K_t)_{t \in T}$
denote a filtration
of $K$.

Let $\phi_p^s = \phi_p(K_s): Z_p(K_s) \rightarrow \HH_p(K_s)$ be the canonical surjection. 
For every $s,t\in T, s < t$ we denote,
\begin{eqnarray}
\label{eqn:tildeM}
\widetilde{M}_p^{s,t}(\mathcal{F}) &=& (\phi_p^s)^{-1}(M_p^{s,t}(\mathcal{F})).
\end{eqnarray}

\begin{definition}[Support of a chain]
\label{def:support}
For $z = \sum_{\sigma \in K^{\sqbracket{p}}}c_\sigma \cdot \sigma \in C_p(K)$, we denote 
$\supp(z) = \{\sigma \in K^{\sqbracket{p}} \mid c_\sigma \neq 0\}$, and call $\supp(z)$ the
\emph{support of $z$}.
\end{definition}


\begin{definition}
\label{def:rep}
Let $b = (s,t; 1) \in \mathbf{B}_p(\mathcal{F})$ be a simple bar
of $\mathcal{F}$ (see Definition~\ref{def:barcode2}).
We define
\begin{eqnarray*}
\Rep(b) &=&  Z_p(K_s) \setminus 
\widetilde{M}^{s,N}_p(\mathcal{F})
\mbox{ if $t = \infty$} \\
&=& \widetilde{M}^{s,t}_p(\mathcal{F}) \setminus \widetilde{M}^{s,t-1}_p(\mathcal{F}) \mbox{ else}.
\end{eqnarray*}
We call $\Rep(b)$ \emph{the set of cycles representing the bar $b$}.
More precisely, 
for $z \in Z_p(K_s)$, $z \in \Rep(b)$, if and only if $z$ represents a non-zero element
in $P^{s,t}_p(\mathcal{F})$.
\end{definition}

We
now arrive at a key definition from \cite{Basu-Cox}.
\begin{definition}[The set of essential simplices associated to a simple bar]
\label{def:essential}
Let $b = (s,t;1) \in \mathbf{B}_p(\mathcal{F})$ be a simple bar
of $\mathcal{F}$.
We define
\begin{equation}
\label{eqn:def:essential}
\Sigma(b) = \bigcap_{z \in \Rep(b)} \supp(z).
\end{equation}
We will call $\Sigma(b)$ \emph{the set of essential simplices of $b$}.
\end{definition}

\begin{remark}
\label{rem:essential-empty}
\begin{figure}
    \centering
    \includegraphics[scale=0.50]{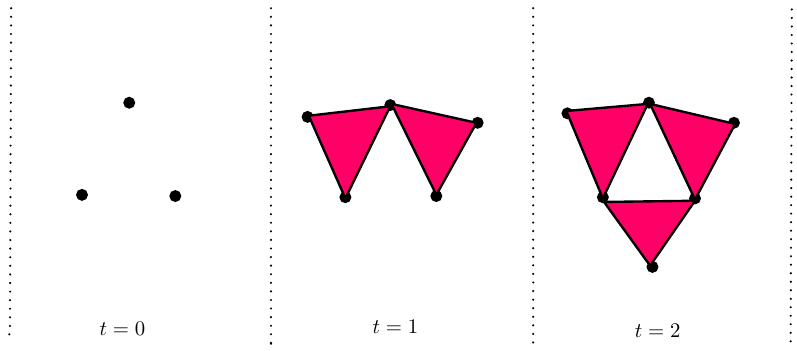}
    \caption{Empty set of essential edges}
    \label{fig:essential-empty}
\end{figure}
Note that the set of essential simplices of a bar can be empty. For example, the unique bar,
$(2,\infty;1) \in \mathbf{B}_1(\mathcal{F})$ of the filtration $\mathcal{F}$ shown in 
Figure~\ref{fig:essential-empty} has no essential simplices (edges). None of the edges 
are indispensable for obtaining a representative of the unique non-zero homology class in dimension one
that is born at $t=2$.
However, this situation cannot occur if the filtration is simplex-wise (see \cite{essential}). 
In that case the last simplex added
at the time a bar is created is always in the set of essential simplices of the bar. The filtration in Figure~\ref{fig:essential-empty} is not a simplex-wise filtration.
\end{remark}

\begin{remark}
   \label{rem:def:essential:efficiency} 
Eqn.~\eqref{eqn:def:essential} does not immediately give an effective way to test membership in $\Sigma(b)$, we will address the question of efficiently testing membership in $\Sigma(b)$ for any generic bar $b$ (see Remark~\ref{rem:efficiency}). 
\end{remark}
\subsection{Example}
\label{subsec:Example}

Before proceeding further we discuss an example (taken from \cite{Basu-Cox}) which illustrates the notion
of harmonic representatives and essential simplices. 
A bar $b$ in the barcode of a filtration is described by a triple $(s,t;\mu)$,
where $s$ denotes the birth time, $t$ the death time and $\mu$ the multiplicity 
(see Definition~\ref{def:barcode2}).

\begin{example}
\label{eg:essential}
\begin{figure}
    \centering
    \includegraphics[scale=0.65]{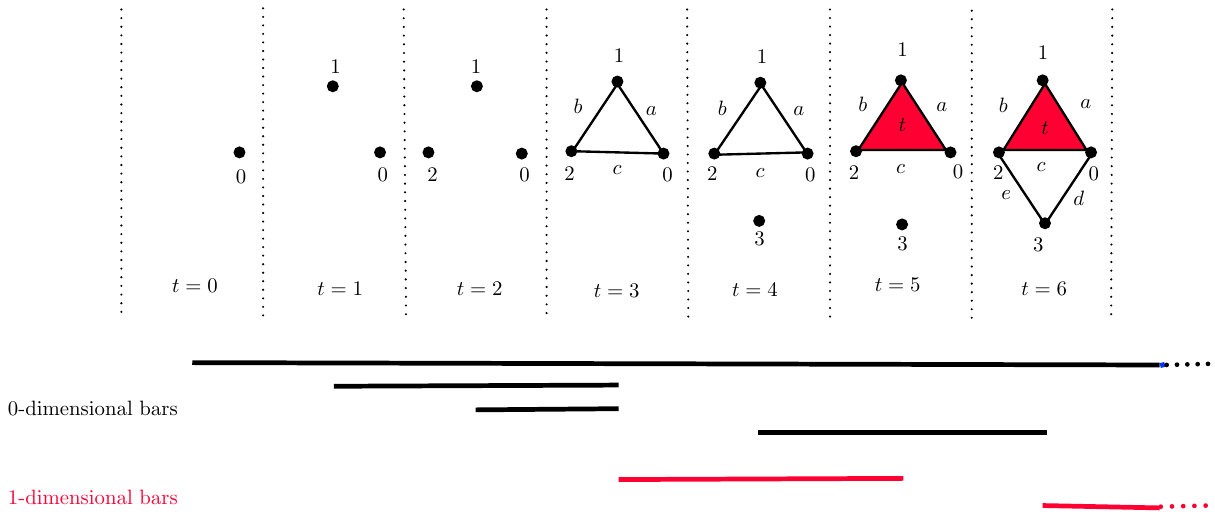}
    \caption{Barcode of a filtration}
    \label{fig:essential}
\end{figure}
Let $K$ be the simplicial complex defined by:
\begin{eqnarray*}
K^{\sqbracket{0}} &=& \{[0],[1],[2],[3] \}, \\
K^{\sqbracket{1}} &=& \{a =[0,1], b= [1,2], c= [0,2], d = [0,3], e= [2,3] \}, \\
K^{\sqbracket{2}} &=& \{ t = [0,1,2] \}.
\end{eqnarray*}
For $p=0,1,2$, we choose the standard inner product on $C_p(K)$ (see \eqref{eqn:standard}).

Let $\FF$ be the following filtration on the $K$:
\[
\emptyset\subset \{0\}\subset\{0,1\}\subset\{0,1,2\}\subset \{0,1,2,a,b,c\}\subset \{0,1,2,3,a,b,c\}
\]
\[
\subset \{0,1,2,3,a,b,c,t\}\subset \{0,1,2,3,a,b,c,d,e,t\}.
\]

For simplicity, we assume that vertex $\{0\}$ is added at time 0, and each complex in the filtration occurs at time 1 greater than the complex preceding it (see Figure~\ref{fig:essential}).
It is clear that all of the bars of the barcode of this filtration are simple.
The corresponding harmonic persistent homology 
subspaces  are listed in the following table. Note that since all the bars are simple, all these
subspaces have dimension $1$.
\begin{center}
\begin{tabular}{|c|c|}
\hline
    \multicolumn{2}{|c|}{$p=0$} \\
    \hline
     $\mathcal{P}_0^{0,\infty}(\mathcal{F})$& $\spn\{0\}$ \\
     \hline
     $\mathcal{P}_0^{1,3}(\FF)$ & $\spn\{1\}$ \\
     \hline
     $\mathcal{P}_0^{2,3}(\FF)$ & $\spn\{2\}$ \\
     \hline
     $\mathcal{P}_0^{4,6}(\FF)$ & $\spn\{3\}$ \\
     \hline
\end{tabular}
\end{center}
\begin{center}
\begin{tabular}{|c|c|}
\hline
    \multicolumn{2}{|c|}{$p=1$} \\
    \hline
     $\mathcal{P}_1^{3,5}(\mathcal{F})$& $\spn\{a+b-c\}$ \\
     \hline
     $\mathcal{P}_1^{6,\infty}(\FF)$ & $\spn\left\{a+b+2c-3d+3e\right\}$ \\
     \hline
\end{tabular}
\end{center}

For $p = 1$, the set of essential  simplices for each bar is listed below.
\begin{eqnarray}
\label{eqn:eg:essential:1}
\Sigma((3,5;1)) &=& \{a,b,c\}, \\
\label{eqn:eg:essential:2}
\Sigma((6,\infty;1)) &=& \{d,e\}.
\end{eqnarray}

\end{example}

We observe that the harmonic representative 
$\mathcal{P}_1^{6,\infty}(\FF)$
of the (generic) bar $(6,\infty;1)$ in Example~\ref{eg:essential} has the property that the absolute value of the coefficients of the essential simplices
$d,e$ of this bar (see \eqref{eqn:eg:essential:2})
are strictly greater than the absolute values of the non-essential ones 
$a,b,c$.
Also the coefficients of the essential simplices $\{a,b,c\}$ in each harmonic representative of the harmonic bar $\mathcal{P}_1^{3,5}(\mathcal{F})$, as well as
those of the essential simplices $\{d,e\}$ 
in each harmonic representative of the harmonic bar $\mathcal{P}_1^{6,\infty}(\FF)$
are equal to each other in absolute value.
 
The importance of the coefficients of various simplices in harmonic representatives (called harmonic weights in \cite{Gurnari-et-al}) in  genomic applications was observed in \emph{loc. cit.}.
We also know from \cite{Basu-Cox} that harmonic representatives maximizes the essential content amongst all  representatives. Taking one step further, 
and in view of the observations made above with respect to the harmonic representatives in Example~\ref{eg:essential},
it is reasonable to make the following (somewhat optimistic) conjecture.

\begin{conjecture}
\label{conj:essential}
Let $K$ be a finite simplicial complex and $\mathcal{F} = (K_t)_{t \in T}$ denote a filtration
of $K$. Suppose $p \geq 0$, and  
let $b = (s,t;1) \in \mathbf{B}_p(\mathcal{F})$ be a simple bar.
Let $z_0 = \sum_{\sigma} c_\sigma\cdot \sigma$
be a harmonic representative of $b$. 
\hide{
Then for $\sigma \in \Sigma(b), \tau \in K^{[p]}\setminus\Sigma(b)$,
\[
|c_\tau| < |c_\sigma|.
\]
}
Then for $\sigma \in \Sigma(b), \tau \in K^{[p]}, \tau \neq \sigma$,
\[
|c_\tau| \leq  |c_\sigma|,
\]
with equality if and only if $\tau \in \Sigma(b)$.
\end{conjecture}

Surprisingly, it turns out that under very mild conditions on the bar $b$, Conjecture~\ref{conj:essential} is true in dimension one.
We prove the following theorem which is the main result of the paper.

\begin{theorem}
\label{thm:main}
Let $\mathcal{F}$ denote a filtration $K_0 \subset K_1 \subset \cdots \subset K_N$ of finite simplicial complexes.  Suppose $p \in \{0,1\}$, and  
let $b = (s,t;1) \in \mathbf{B}_p(\mathcal{F})$ be a generic bar.
Let $z_0 = \sum_{\sigma} c_\sigma\cdot \sigma$
be a harmonic representative of $b$. 
\hide{
Then for $\sigma \in \Sigma(b), \tau \in K^{[p]}\setminus\Sigma(b)$,
\[
|c_\tau| < |c_\sigma|.
\]
}
Then for $\sigma \in \Sigma(b), \tau \in K^{[p]}, \tau \neq \sigma$,
\[
|c_\tau| \leq |c_\sigma|,
\]
with equality if and only if $\tau \in \Sigma(b)$.
\end{theorem}

\begin{remark}
\label{rem:restrictions}
    The restriction on $p$ as well as
    the genericity assumption are both required.
    We construct a sequence (one for each $p$) of counter-examples to Conjecture~\ref{conj:essential} for $p>1$ (see Section~\ref{sec:counterexample}). For $p=1$, we
    show that Property~\eqref{itemlabel:generic:1} in the definition of genericity is needed for the inequality
    in Theorem~\ref{thm:main} to be non-vacuous,
     and also provide counter-examples if Property~\eqref{itemlabel:generic:2} in the definition of genericity is violated (see Section~\ref{sec:appendix}).
     Also note that the case $p=0$ is almost vacuous. So the main strength of Theorem~\ref{thm:main} lies in the case $p=1$.
\end{remark}

\subsection{Prior and related work}
We briefly review in this section prior and related work on persistent harmonic homology.
Persistent harmonic cohomology has being mentioned before \cite{talk:Lieutier} (see also \cite{Chen-et-al-2021}),
but the 
emphasis is more on manifold-learning rather than on simplicial filtrations. Memoli et al.\  \cite{Memoli-et-al} also studies persistent homology groups,
defining them in terms of the Laplace operator (see Remark~\ref{rem:harmonic}) 
and gives efficient
algorithms for computing them. They also establish interesting connections with spectral graph theory and
prove certain stability results on the eigenvalues of the Laplace operator when applied to a simplicial
filtration.
More recently, the paper \cite{gulen2025grassmannianpersistencediagramsspecial} defines  Grassmannian persistence diagrams to persistent Laplacians and establishes an isomorphism with harmonic barcodes. It would interesting to interpret the results of the current paper in their language. 
Harmonic homology as attracted the attention of the quantum information theorists. 
Super-polynomial speed ups for quantum algorithms for computing harmonic homology subspaces appear in \cite{Hayakawa-et-al-2026} (see also \cite{gyurik-et-al}).

\medskip
The rest of the paper is organized as follows. In Section~\ref{sec:main}, we prove Theorem~\ref{thm:main}. In Section~\ref{sec:counterexample} we construct counterexamples  to Conjecture~\ref{conj:essential} in two steps.
In Subsection~\ref{subsec:counter:2}, we first construct a counter-example with $p=2$ and in Subsection~\ref{subsec:counter:p}, we extend this counterexample to all $p \geq 2$. In Section~\ref{sec:appendix} we 
provide counter-examples if the genericity hypothesis is violated. Finally, in Section~\ref{sec:conclusion} we discuss some future work.

\section{Proof of Theorem~\ref{thm:main}}
\label{sec:main}
In this section we prove Theorem~\ref{thm:main}. The proof involves several steps. In Subsection~\ref{subsec:relative} we reduce the 
proof of Theorem~\ref{thm:main}
to the special case of filtrations of length two satisfying an extra condition (Theorem~\ref{thm:simplified}). 
For filtrations of length two and for generic bars, we give a simpler characterization of essential simplices (see Definition~\ref{def:relative-essential}) and prove that it agrees
with the original definition (Lemma~\ref{lem:simplified}). This reduces the proof of Theorem~\ref{thm:main} to proving dominance of the coefficients of
the relatively essential simplices in harmonic representatives in dimension one, which we formulate as Theorem~\ref{thm:simplified}.
In Subsection~\ref{subsec:simplified} we prove Theorem~\ref{thm:simplified}. 
(Note that Theorem~\ref{thm:simplified} is the core mathematical result of the paper; Theorem~\ref{thm:main} follows from the reduction established in Subsection~\ref{subsec:relative}.)

\subsection{Reduction to length two filtrations}
\label{subsec:relative}
In this section, we reduce the general case of Theorem~\ref{thm:main}
to studying certain filtrations $K' \subset K$ of length two satisfying 
certain condition (see Eqn. \eqref{eqn:two-step}). 
We define a notion of \emph{relative essential simplices} (Definition~\ref{def:relative-essential}) and prove (Lemma~\ref{lem:simplified}) that under assumption of genericity of bars (Definition~\ref{def:generic}), 
this notion agrees with the previous definition of essential simplices (Definition~\ref{def:essential}) (with $K=K_s$ and $K' = K_{s-1}$ for a generic bar $b = (s,t;1)$). 
The alternative definition happens to be a much simpler characterization of being essential than the original definition, and is easier to verify and use (see Remark~\ref{rem:efficiency}). 
As a result it makes possible the use of certain graph Laplacian arguments that play a key role in the proof of Theorem~\ref{thm:simplified}. 

The following proposition is the starting point of the reduction mentioned above.
\begin{proposition}
    \label{prop:simplified}
    Let $K$ be a finite simplicial complex and  
let $\mathcal{F}$ denote a finite filtration
$K_0 \subset \cdots \subset K_N= K$.
By convention we will assume that $K_s = \emptyset$ for $s <0$ and 
$K_t = K$ for $t \geq N$.
    Let $s \in T$, $t \in T \cup \{\infty\}$ and $s < t$.
Suppose that the bar $b = (s,t;1) \in \mathbf{B}_p(\mathcal{F})$ is generic.
Then, 
    \[
    \mathcal{P}_p^{s,t}(\mathcal{F}) = \mathcal{H}_p(K_s) \cap (\mathrm{Im} \; \mathfrak{i}_p^{s-1,s})^\perp,
    \]
and 
\[
\dim \mathcal{P}_p^{s,t}(\mathcal{F}) = 1.
\]
\end{proposition}

\begin{proof}
We have a filtration of subspaces, 
\[
\HH_p(K_s) \supset M^{s,N}_p(\mathcal{F}) \supset \cdots  \supset M^{s,t}_p(\mathcal{F})  \supset M^{s,t-1}_p(\mathcal{F}) \supset \cdots \supset 
M^{s,s+1}_p(\mathcal{F}) \supset M^{s,s}_p(\mathcal{F}) = \HH_p^{s-1,s}(\mathcal{F}).
\]
We separate the two cases, $t = \infty$ and $t \neq \infty$.
\begin{enumerate}
    \item 
    \label{itemlabel:proof:prop:simplified:1}
    Case $t = \infty$. In this case all the inclusions 
    \[
    M^{s,N}_p(\mathcal{F}) \supset \cdots  
    \cdots \supset 
M^{s,s+1}_p(\mathcal{F}) \supset M^{s,s}_p(\mathcal{F}) = \HH_p^{s-1,s}
    \]
    are equalities. Otherwise let $t', N \geq t' > s$ be the largest index such that 
    \[
    M^{s,t'}_p(\mathcal{F})  \supsetneq  M^{s,t'-1}_p(\mathcal{F}).
    \]
    But this implies that $\mu_p^{s,t'} > 0$, and since $t' \neq \infty$,
    this will contradict 
    Property~\eqref{itemlabel:generic:2} in Definition~\ref{def:generic} 
    and the fact that $b$ is assumed to be generic.
    It now follows from Property~\eqref{itemlabel:generic:1} in Definition~\ref{def:generic} and the definition of
    $\mathcal{P}_p^{s,\infty}(\mathcal{F})$ (Definition~\ref{def:harmonic-barcode})  that
    \[
    \mathcal{P}_p^{s,\infty}(\mathcal{F}) = \mathcal{H}_p(K_s) \cap (\mathrm{Im} \; \mathfrak{i}_p^{s-1,s})^\perp,
    \]
and 
\[
\dim \mathcal{P}_p^{s,\infty}(\mathcal{F}) = 1.
\]
    in this case.
    \item 
     \label{itemlabel:proof:prop:simplified:2}
    Case $t \neq \infty$. In this case we claim that 
    all the containments other than $M^{s,t}_p(\mathcal{F})  \supset M^{s,t-1}_p(\mathcal{F})$, in the sequence
    \[
\HH_p(K_s) \supset M^{s,N}_p(\mathcal{F}) \supset \cdots  \supset M^{s,t}_p(\mathcal{F})  \supset M^{s,t-1}_p(\mathcal{F}) \supset \cdots \supset 
M^{s,s+1}_p(\mathcal{F}) \supset M^{s,s}_p(\mathcal{F}) = \HH_p^{s-1,s}(\mathcal{F}).
\]
are equalities.
Otherwise let $t', s < t' \leq N, t' \neq t $ be an index such that 
    \[
    M^{s,t'}_p(\mathcal{F})  \supsetneq  M^{s,t'-1}_p(\mathcal{F}).
    \]
    But this implies that $\mu_p^{s,t'} > 0$, and since $t' \neq t$,
    this will contradict 
    Property~\ref{itemlabel:generic:2} in Definition~\ref{def:generic} 
    and the fact that $b$ is assumed to be generic.

    It now follows from Property~\eqref{itemlabel:generic:1} in Definition~\ref{def:generic} and the definition of
    $\mathcal{P}_p^{s,t}(\mathcal{F})$ (Definition~\ref{def:harmonic-barcode})  that
    \[
    \mathcal{P}_p^{s,t}(\mathcal{F}) = \mathcal{H}_p(K_s) \cap (\mathrm{Im} \; \mathfrak{i}_p^{s-1,s}))^\perp,
    \]
and 
\[
\dim \mathcal{P}_p^{s,t}(\mathcal{F}) = 1.
\]
in this case.
\end{enumerate}

\end{proof}


Let $K$ be a finite simplicial complex and $K' \subset K$ a sub-complex,
and $i:K' \subset K$ the inclusion. Suppose that
\begin{equation}
\label{eqn:two-step}
\dim \HH_p(K)/ i_p(\HH_p(K')) = 1.
\end{equation}

Now let \(h\in C_p(K), h \neq 0\) be a nonzero chain spanning the one-dimensional space
\[
\Harm_p(K)\cap \Harm_p(K')^\perp.
\]

\begin{remark}
\label{rem:K'}
We remark that under the assumption $\dim \HH_p(K)/ i_p(\HH_p(K')) = 1$,
\begin{equation}
\label{eqn:rem:K'}
\Harm_p(K)\cap \Harm_p(K')^\perp = \Harm_p(K)\cap \mathfrak{i}_p(\Harm_p(K'))^\perp.
\end{equation}

To see this observe that for $k' \in \Harm_p(K')$, $\mathfrak{i}_p(k') = \mathrm{proj}_{B_p(K)^\perp}(k')$. Thus, $k' - \mathfrak{i}_p(k') \in B_p(K)$.

Let $h_1 \in \Harm_p(K)\cap \Harm_p(K')^\perp$.
Since $h_1 \in \Harm_p(K)$, $h_1 \in B_p(K)^\perp$.
Therefore, $\la h_1, k'\ra = \la h_1, \mathfrak{i}_p(k') \ra$.

Hence,
\[
h_1 \perp \mathfrak{i}_p(\Harm_p(K')) \Leftrightarrow h_1 \perp \Harm_p(K').
\]
\end{remark}

Let
\begin{equation}
\label{eqn:W}
W_p:=B_p(K)+Z_p(K').
\end{equation}
Clearly,
\[
W_p \subseteq Z_p(K).
\]

Thus
$
h\in\Harm_p(K),
h\perp \Harm_p(K')
$.

\begin{definition}[Relative essential simplex]
\label{def:relative-essential}
We call a simplex \(\sigma\) of \(K^{[p]}\) \emph{relatively essential}
with respect to \(h\) if
\[
\langle h+w,\sigma\rangle\neq 0
\qquad
\text{for every }w\in W_p.
\]
Equivalently,
\[
\langle h+b+z',\sigma\rangle\neq 0
\qquad
\text{for every }b\in B_p(K),\ z'\in Z_p(K').
\]

We denote by  \(S_E(h)\)  the set of relatively essential simplices with respect to $h$, and by
\(S_N(h) = K^{[p]} \setminus S_E(h)\). 
\end{definition}

Here is an alternative characterization of being relatively essential. 
\begin{lemma}[Alternative characterization of relative essentiality]
\label{lem:linear}
For \(\sigma \in K^{[p]}\), the following are equivalent (see Definition~\ref{def:relative-essential}):
\[
\sigma\in S_E(h)
\]
and
\[
\langle h,\sigma\rangle\neq 0
\quad\text{and}\quad
\sigma\perp W_p.
\]
\end{lemma}

\begin{proof}
Suppose first that \(\sigma\in S_E(h)\). Taking \(w=0\) gives
$
\langle h,\sigma\rangle\neq 0
$.
If there existed \(w_0\in W_p\) such that
$
\langle w_0,\sigma\rangle\neq 0
$,
then, for
\[
t=-\frac{\langle h,\sigma\rangle}{\langle w_0,\sigma\rangle},
\]
we would have
$
\langle h+t w_0,\sigma\rangle=0
$,
contradicting relative essentiality. Therefore \(\sigma\perp W_p\).

Conversely, suppose
\[
\langle h,\sigma\rangle\neq 0
\qquad\text{and}\qquad
\sigma\perp W_p.
\]
Then for every \(w\in W_p\),
\[
\langle h+w,\sigma\rangle
=
\langle h,\sigma\rangle+\langle w,\sigma\rangle
=
\langle h,\sigma\rangle
\neq 0.
\]
Hence \(\sigma\in S_E(h)\).
\end{proof}


We also have:
\begin{lemma}
\label{lem:orthogonality}
\[
h\perp W_p.
\]
\end{lemma}

\begin{proof}
Since
$
h\in\Harm_p(K)
$,
we have
$
h\perp B_p(K)
$.
It remains to show that
$
h\perp Z_p(K')
$.

Let
$
z'\in Z_p(K')
$.
By the Hodge decomposition on \(K'\) (see for example \cite{Lim2020}), there exist
$
b'\in B_p(K'),
k'\in\Harm_p(K')
$,
such that
$
z'=b'+k'
$.

Since
$
B_p(K')\subseteq B_p(K)
$,
we have
$
h\perp b'
$.

Since
$
h\perp \Harm_p(K')
$,
we have
$
h\perp k'
$.

Therefore
\[
\langle h,z'\rangle
=
\langle h,b'\rangle+\langle h,k'\rangle
=
0.
\]

Hence
$
h\perp Z_p(K')
$.

Since
$
W_p=B_p(K)+Z_p(K')
$,
we conclude that
$
h\perp W_p
$.
\end{proof}

\begin{lemma}
\label{lem:W}
    Suppose that $b=(s,t;1) \in B_p(\mathcal{F})$ be a generic bar in a finite filtration $\mathcal{F}$. 
    Let $K = K_s, K' = K_{s-1}$ and $W_p$ be the subspace defined in 
    Eqn.~\eqref{eqn:W}.
    Then,
    \begin{eqnarray*}
    W_p &=& 
    \widetilde{M}^{s,N}_p(\mathcal{F}) \text{ if $t = \infty$},\\
    &=& \widetilde{M}^{s,t-1}_p(\mathcal{F}) \text{ if $t \neq \infty$}.
    \end{eqnarray*}
\end{lemma}

\begin{proof}
   In the case $t = \infty$. It follows from  
   Eqn. \eqref{eqn:tildeM}
   and Case~\ref{itemlabel:proof:prop:simplified:1} in the proof of 
    Proposition~\ref{prop:simplified} that in this case,
    \begin{equation}
    \label{eqn:proof:lem:W:1}
    \widetilde{M}^{s,N}_p(\mathcal{F}) 
    = (\phi_p^s)^{-1}(\HH^{s-1,s}_p(\mathcal{F})) = (\phi_p^s)^{-1}(i_p^{s-1,s}(\HH_p(K_{s-1}))).
    \end{equation}

    If $t \neq \infty$, 
    then  it follows from 
    Eqn. \eqref{eqn:tildeM}
    and Case~\ref{itemlabel:proof:prop:simplified:2} in the proof of 
    Proposition~\ref{prop:simplified} that
    \begin{equation}
    \label{eqn:proof:lem:W:2}
    \widetilde{M}^{s,t-1}_p(\mathcal{F}) = 
    (\phi_p^s)^{-1}(M^{s,t-1}_p(\mathcal{F})) = (\phi_p^s)^{-1}(i_p^{s-1,s}(\HH_p(K_{s-1}))).
    \end{equation}
    
    We also have the following commutative diagram.
    \begin{equation}
    \label{eqn:diagram}
      \xymatrix{
    B_p(K_{s-1}) \ar@{^{(}->}[r]\ar@{^{(}->}[d] & B_p(K_s) \ar@{^{(}->}[d] \\
    Z_p(K_{s-1}) \ar@{^{(}->}[r]\ar@{->>}[d]^{\phi_p^{s-1}} & Z_p(K_s) \ar@{->>}[d]^{\phi_p^s}\\
    \HH_p(K_{s-1})\ar[r]^{i_p^{s-1,s}} & \HH_p(K_s)
    }  
    \end{equation}

In view of \eqref{eqn:proof:lem:W:1} and \eqref{eqn:proof:lem:W:2}
it suffices to prove that 
\[
W_p = Z_p(K_{s-1}) + B_p(K_s) = (\phi_p^s)^{-1}(i_p^{s-1,s}(\HH_p(K_{s-1}))).
\]
We first prove that
\[
W_p = Z_p(K_{s-1}) + B_p(K_s) \subset (\phi_p^s)^{-1}(i_p^{s-1,s}(\HH_p(K_{s-1}))).
\]

Suppose that $z \in Z_p(K_{s-1})$. Then,
\[
i^{s-1,s}_p \circ \phi^{s-1}_p(z) = \phi^{s}_p(z)
\]
from the commutativity of \eqref{eqn:diagram}, which shows that 
\[
z \in (\phi_p^s)^{-1}(i_p^{s-1,s}(\HH_p(K_{s-1}))).
\]
This shows that 
\[
Z_p(K_{s-1}) \subset (\phi_p^s)^{-1}(i_p^{s-1,s}(\HH_p(K_{s-1}))).
\]

We now show that $B_p(K_{s}) \subset (\phi_p^s)^{-1}(i_p^{s-1,s}(\HH_p(K_{s-1})))$. Clearly, if $b \in B_p(K_s)$, then $\phi_p^s(b) = 0$, and 
$0 \in i_p^{s-1,s}(\HH_p(K_{s-1}))$. Therefore, 
$b \in (\phi_p^s)^{-1}(i_p^{s-1,s}(\HH_p(K_{s-1})))$.

Conversely, suppose that 
\[
z \in (\phi_p^s)^{-1}(i_p^{s-1,s}(\HH_p(K_{s-1}))).
\]

Then, 
\[
\phi_p^s(z) \in (i_p^{s-1,s}(\HH_p(K_{s-1}))).
\]
Using a one-step diagram chase in \eqref{eqn:diagram} there exists $z' \in Z_p(K_{s-1})$ such that $\phi_p^s(z) = \phi_p^s(z')$. But then $z - z' \in B_p(K_s)$, implying that $z \in W_p$.
\end{proof}

\begin{lemma}
\label{lem:tildeM-harmonictildeM}
With the same notation as above:
\[
\widetilde{M}^{s,t}_p(\mathcal{F}) = (\mathfrak{f}_p(K_s)\circ \phi_p^s)^{-1}(\mathcal{M}_p^{s,t}(\mathcal{F})).
\]
\end{lemma}

\begin{proof}
    Follows from Definitions~\ref{def:barcode}, \ref{def:harmonic-barcode},
    Eqn.~\eqref{eqn:tildeM} and Proposition~\ref{prop:functorial}.
\end{proof}

\begin{lemma}
\label{lem:h}
    With the same hypothesis as Lemma~\ref{lem:W}, suppose that
    $h \in \mathcal{P}^{s,t}_p(\mathcal{F}) \subset C_p(K_s)$ is a harmonic representative of $b$.  
    Then
    \begin{eqnarray*}
    \mathrm{span}(h) \oplus  W_p &=& Z_p(K_s) \text{ if $t = \infty$},\\
        &=& \widetilde{M}^{s,t}_p(\mathcal{F}) \text{ else}.
    \end{eqnarray*}
\end{lemma}

\begin{proof}
\begin{enumerate}
    \item Case $t=\infty$. 
    Using the fact that $\mu^{s,\infty}_p(\mathcal{F}) = 1$, we have that 
    \begin{equation}
    \label{eqn:proof:lem:h:1}
    \dim \mathcal{H}_p(K_s)/\mathcal{M}^{s,N}_p(\mathcal{F})=
    1.
    \end{equation}

    Now observe that
    $\mathcal{M}_p^{s,N}(\mathcal{F}) \subset  \mathcal{H}_p(K_s)$, and 
    that the map 
    \[
      \mathfrak{f}_p(K_s) \circ \phi_p(K_s): Z_p(K_s) \rightarrow \mathcal{H}_p(K_s)
    \]
    is the orthogonal projection on to the subspace 
    $\mathcal{H}_p(K_s) \subset Z_p(K_s)$ (see Proposition~\ref{prop:proj}).
    Now $Z_p(K_s)$ (resp. $\widetilde{M}^{s,N}_p(\mathcal{F})$) is the inverse image 
    of $\mathcal{H}_p(K_s)$ (resp. $\mathcal{M}^{s,N}_p(\mathcal{F})$) under $\mathfrak{f}_p(K_s) \circ \phi_p(K_s)$ (Lemma~\ref{lem:tildeM-harmonictildeM}). The restriction $\mathfrak{f}_p(K_s) \circ \phi_p(K_s)\restriction_{\mathcal{H}_p(K_s)}$ is the identity map; hence,
    \[
    h \in Z_p(K_s), h \not\in  W_p = \widetilde{M}^{s,N}_p(\mathcal{F}),
    \]
    since $h \in \mathcal{H}_p(K_s) \cap \mathcal{M}^{s,N}(\mathcal{F})^\perp$.
    Since,
    $\mathcal{M}^{s,N}_p(\mathcal{F})$ is a codimension one subspace of $\mathcal{H}_p(K_s)$ (using \eqref{eqn:proof:lem:h:1}),
    $\widetilde{M}^{s,N}_p(\mathcal{F})$ is also a codimension one subspace of $Z_p(K_s)$ (pull-back under orthogonal projection), and we obtain that
    $Z_p(K_s) = \mathrm{span}(h) + W_p$, and the sum is direct since $h \not\in W_p$.
    \item Case $t \neq \infty$. 
    Using the fact that $\mu^{s,t}_p(\mathcal{F}) = 1$, we have that 
    \begin{equation}
    \label{eqn:proof:lem:h:2}
    \dim M^{s,t}(\mathcal{F})/M^{s,t-1}_p(\mathcal{F}) = 
    \dim \mathcal{M}^{s,t}(\mathcal{F})/\mathcal{M}^{s,t-1}_p(\mathcal{F})=
    1.
    \end{equation}

    Now observe that
    $\mathcal{M}_p^{s,t-1}(\mathcal{F}) \subset \mathcal{M}_p^{s,t}(\mathcal{F}) \subset \mathcal{H}_p(K_s)$, and 
    that the map 
    \[
    \mathfrak{f}_p(K_s) \circ \phi_p(K_s): Z_p(K_s) \rightarrow \mathcal{H}_p(K_s)
    \]
    is the orthogonal projection on to the subspace 
    $\mathcal{H}_p(K_s) \subset Z_p(K_s)$ (see Proposition~\ref{prop:proj}).
    Now $\widetilde{M}^{s,t}_p(\mathcal{F})$ (resp. $\widetilde{M}^{s,t-1}_p(\mathcal{F})$) is the inverse image 
    of $\mathcal{M}^{s,t}_p(\mathcal{F})$ (resp. $\mathcal{M}^{s,t-1}_p(\mathcal{F})$) under $\mathfrak{f}_p(K_s) \circ \phi_p(K_s)$ (Lemma~\ref{lem:tildeM-harmonictildeM}). The restriction $\mathfrak{f}_p(K_s) \circ \phi_p(K_s)\restriction_{\mathcal{M}^{s,t}_p(\mathcal{F})}$ is the identity map; hence,
    \[
    h \in \widetilde{M}^{s,t}_p(\mathcal{F}), h \not\in  W_p = \widetilde{M}^{s,t-1}_p(\mathcal{F}),
    \]
    since $h \in \mathcal{M}^{s,t}(\mathcal{F}) \cap \mathcal{M}^{s,t-1}(\mathcal{F})^\perp$.
    Since,
    $\mathcal{M}^{s,t-1}_p(\mathcal{F})$ is a codimension one subspace of $\mathcal{M}^{s,t}_p(\mathcal{F})$ (using \eqref{eqn:proof:lem:h:2}),
    $\widetilde{M}^{s,t-1}_p(\mathcal{F})$ is also a codimension one subspace of $\widetilde{M}^{s,t}_p(\mathcal{F})$ (pull-back under orthogonal projection), and we obtain that
    $\widetilde{M}^{s,t}(\mathcal{F}) = \mathrm{span}(h) + W_p$ and the sum is direct since $h \not\in W_p$.
\end{enumerate}
\end{proof}

We have a general definition of the set of essential simplices, $\Sigma(b)$, for any simple bar $b$ (Definition~\ref{def:essential}). We then defined the set of relatively essential simplices, $S_E(h)$, of a harmonic representative $h$ of a simple bar in a filtration $K' \subset K$ of length two (Definition~\ref{def:relative-essential}). 
The next lemma shows that these two sets coincide in the case when $h$ is a harmonic representative of a \emph{generic} (Definition~\ref{def:generic}) bar $b$.

\begin{lemma}
\label{lem:simplified}
    Suppose that $b=(s,t;1) \in B_p(\mathcal{F})$ be a generic bar in a finite filtration $\mathcal{F}$. 
    Let $h \in \mathcal{P}^{s,t}_p(\mathcal{F}) \subset C_p(K_s)$ be a harmonic representative of $b$.  Let $K = K_s, K' = K_{s-1}$.
    Then, 
    \begin{eqnarray*}
        \Sigma(b) &=& S_E(h).
    \end{eqnarray*}
\end{lemma}
\begin{proof}

We first prove that $S_E(h) \subset \Sigma(b)$. Suppose that $\sigma \in S_E(h)$, and $z \in \Rep(b)$. Then, using Lemma~\ref{lem:h}, there is a unique expression 
$z = z_1 + z_2$, $z_1 = c\cdot h\in \mathrm{span}(h)$, $z_2 \in W_p$.
Moreover, since $\Rep(b) \cap W_p = \emptyset$, $c \neq 0$.
Since $\sigma \in S_E(h)$, $\langle c^{-1} z, \sigma\rangle \neq 0$ which implies that $\sigma \in \supp(z)$. This proves that $\sigma \in \Sigma(b)$.
    
We break the proof of $\Sigma(b) \subset  S_E(h)$ into two cases.
\begin{enumerate}
\item Case $t = \infty$.
Let $\sigma\in \Sigma(b)$. Then,
    $\sigma \in \supp(z)$ for every $z \in \Rep(b) = Z_p(K_s) \setminus \widetilde{M}^{s,N}(\mathcal{F})$. Let $w \in W_p = \widetilde{M}^{s,N}(\mathcal{F})$. Then, $h + w \in \Rep(b)$ using Lemma~\ref{lem:h}, and 
    $\sigma \in \supp(h+w) \Leftrightarrow \langle h+w,\sigma\rangle \neq 0$.
    So $\sigma \in S_E(h)$.

\item Case $t \neq \infty$. Let $\sigma\in \Sigma(b)$. Then,
    $\sigma \in \supp(z)$ for every $z \in \Rep(b) = \widetilde{M}^{s,t}(\mathcal{F}) \setminus \widetilde{M}^{s,t-1}(\mathcal{F})$. Let $w \in W_p = \widetilde{M}^{s,t-1}(\mathcal{F})$. Then, $h + w \in \Rep(b)$ using Lemma~\ref{lem:h}, and 
    $\sigma \in \supp(h+w) \Leftrightarrow \langle h+w,\sigma\rangle \neq 0$.
    So $\sigma \in S_E(h)$.
\end{enumerate}
\end{proof}

\begin{remark}
\label{rem:efficiency}
    We note that Lemmas~\ref{lem:linear} and \ref{lem:simplified} can serve as a basis for an \emph{efficient algorithm for computing the set $\Sigma(b)$} for a generic bar $b$. The original definition of $\Sigma(b)$ involves taking the intersection of the supports of all cycles in $\Rep(b)$ which is an infinite set. On the other hand Lemma~\ref{lem:linear} reduces the membership question of a simplex $\sigma$ in $S_E(h)$ to a finite dimensional linear algebra problem. Lemma~\ref{lem:simplified} then implies that for a generic bar $b$ with
     harmonic representative $h$, $\sigma \in S_E(h) \Leftrightarrow \sigma \in \Sigma(b)$. An implementation of the algorithm based on the above observation will be a very useful addition to any harmonic persistent homology software package. 
\end{remark}

In view of Proposition~\ref{prop:simplified} and Lemma~\ref{lem:simplified}, in order to prove Theorem~\ref{thm:main} it suffices to prove the following theorem.

\begin{theorem}
\label{thm:simplified}
Let \(K\) be a finite simplicial complex and \(K'\subseteq K\) a subcomplex.
Assume
\begin{equation}
\label{eqn:thm:simplified}
\dim_{\mathbb R} H_1(K)/i_1(H_1(K'))=1.
\end{equation}
Let $h$ be a non-zero element of \(\mathcal H_1(K)\) spanning
\(\mathcal H_1(K)\cap \mathfrak{i}_1(\mathcal H_1(K'))^\perp\).
\hide{
Then for every \(\sigma\in S_E(h)\) and every \(\tau\in S_N(h)\),
\[
|\langle h,\tau\rangle|<|\langle h,\sigma\rangle|.
\]
}
Then for every \(\sigma\in S_E(h)\) and every \(\tau\in K^{[1]}, \tau \neq \sigma\),
\[
|\langle h,\tau\rangle| \leq |\langle h,\sigma\rangle|,
\]
with equality if and only if $\tau \in S_E(h)$.
\end{theorem}

\begin{remark}
    Note that Theorem~\ref{thm:simplified} only needs assumption \eqref{eqn:thm:simplified} and does not need the notion of genericity of bars. Later we will use Theorem~\ref{thm:simplified} to prove Theorem~\ref{thm:main}, and the genericity assumption will enter through Proposition~\ref{prop:simplified} and Lemma~\ref{lem:simplified}.
\end{remark}

\subsection{Proof of Theorem~\ref{thm:simplified}}
\label{subsec:simplified}

In this section we prove Theorem~\ref{thm:simplified}. 
We use the notation and hypothesis of Theorem~\ref{thm:simplified} throughout this section.
The proof uses the fact that normalized harmonic representatives are minimum-energy flows with prescribed boundary. Such flows are graph Laplacian potentials. The maximum principle then converts coefficient estimates into cut estimates.
We summarize below the key steps which implements the above strategy. 
\begin{enumerate}[label=\arabic*.]

\item
We characterize the normalized harmonic representative ($h$ multiplied by $\frac{1}{\la h,\sigma \ra}$), as the unique norm-minimizer of a certain translate $\mathcal{A}_\sigma$ of $W_1$ in $C_1(K)$.
\item
Minimizing norm over $\mathcal{A}_\sigma$ is then shown to be equivalent to minimizing 
norm over all of $C_1(G')$ (where $G'$ is the $1$-skeleton of $K$ minus the edge $\sigma$) subject to the constraint $\partial_1 z = \partial_1 \sigma$.
\item
We prove that for such a minimizer $z$, $|\la z , \tau \ra| < 1$ for any edge $\tau$ that is not a bridge in $G'$
(Lemma~\ref{lem:graph}). 
\item
Finally for edges $\tau \in S_N(h), \la h,\tau\ra \neq 0$, we show that $\tau$ cannot be a bridge edge (Lemma~\ref{lem:bridge}), which is sufficient to complete the proof.  
\end{enumerate}

In what follows we are going to identify for any finite simplicial complex $K$,
the $i$-th chain group $C_i(K)$, with its dual vector space $C^i(K)$, via the 
isomorphism 
\[
C_i(K) \rightarrow C^i(K), \qquad \varphi \longmapsto \la \cdot, \varphi \ra.
\]

We need the following result which is well-known in the theory of electrical networks where it is mostly expressed in terms of electrical currents and resistance (see for example \cite{Doyle-Snell-book,Lyons-Peres-book}). We avoid the electrical terminology and include a proof for the sake of completeness.

\begin{lemma}[Minimum-norm prescribed-boundary chains]
\label{lem:graph}
Let \(G\) be a finite oriented graph. Let \(\sigma\) be an oriented edge of
\(G\), with
$
\partial\sigma=b-a
$,
and let
$
G'=G\setminus\{\sigma\}
$.
Suppose 
\(x\in C_1(G')\) is the unique minimum-norm solution of
$
\partial_{G'} x=a-b
$,
assuming the solution set is non-empty.
Then, for every edge \(\tau\) of \(G'\) that is not a bridge separating \(a\)
from \(b\),
$
|\langle x,\tau\rangle|<1
$.
In particular, if \(\tau\) lies in a cycle of \(G'\), then
$
|\langle x,\tau\rangle|<1
$.
\end{lemma}

\begin{proof}
Let
$D$ denote  $\partial_{G'}:C_1(G')\to C_0(G')
$
(the boundary map of \(G'\)). 
The assumption that there exists $x$ such that
$Dx=a-b$, implies in particular that $[a-b] \in B_0(G')$, and so 
$[a] = [b] \in \HH_0(G')$. Hence,
$a,b$ belong to the same connected component of $G'$.

The affine space of solutions is
$
\{u\in C_1(G'):Du=a-b\}
$.
The unique minimum-norm solution \(x\) is orthogonal to the homogeneous
solution space
$
\ker D
$.
Indeed, if \(c\in\ker D\), then \(x+\lambda c\) is again a solution for every
\(\lambda\in\R\). Since \(x\) minimizes the norm, the function
$
\lambda\mapsto \|x+\lambda c\|^2
$
has derivative zero at \(\lambda=0\), so
$
\langle x,c\rangle=0
$,
Therefore
$
x\in(\ker D)^\perp
$.

Since we are in a finite-dimensional inner product space,
$
(\ker D)^\perp=\im D^*
$.
Hence there exists a function
$
\varphi\in C_0(G')
$
such that
$
x=D^*\varphi
$.

For an oriented edge \(e:u\to v\), this means
$
\langle x,e\rangle
=
\varphi(v)-\varphi(u)
$.

Now let \(\tau:u\to v\) be an edge of \(G'\). If
$
\langle x,\tau\rangle=0
$,
then there is nothing to prove. So assume
$
\langle x,\tau\rangle\neq 0
$.

After reversing the orientation of \(\tau\) if necessary, we may assume
$
\langle x,\tau\rangle>0
$.
Thus
$
\varphi(v)-\varphi(u)>0.
$

Choose a real number \(c\) such that
$
\varphi(u)<c<\varphi(v)
$,
and define the vertex subset
\[
A=\{p\in V(G'):\varphi(p)>c\}.
\]

Then \(v\in A\) and \(u\notin A\), so \(\tau\) crosses the cut
$
A\mid A^c
$.

We claim that \(a\in A\) and \(b\notin A\).

To see this, observe that
$
Dx=a-b
$,
and 
\(x=D^*\varphi\). So 
$
DD^*\varphi=a-b
$.
At every vertex \(p\neq a,b\), we have
$
(DD^*\varphi)(p)=0
$.
Equivalently, \(\varphi(p)\) is the average of the values of \(\varphi\) on
the neighbors of \(p\)
(The operator \(DD^*\) is the graph Laplacian.) 

Let $C$ be the connected component of $G'$ containing $a,b$, and let $S$
$\operatorname{argmax}$ set of $\varphi\restriction_C$.
Then for all vertices $p$ of $C$, with $p\neq a,b$, $\varphi(p)$ equals the average value of the neighbors and so $S$ is closed under passing to neighbors at such vertices. Since $(DD^*\varphi)(b) = -1$, $\varphi(b)$ is strictly below the average of its neighbors, $b \not\in S$.  This implies that $a \in S$, and so $\max_C \varphi = \varphi(a)$.
By a symmetric argument $\min_C \varphi = \varphi(b)$.

Since \(\varphi(v)>\varphi(u)\), the component containing \(\tau\) is not a
component on which \(\varphi\) is constant.
On any  component $C$ of $G'$ such that $a,b \not\in C$, we have
$DD^*\varphi\restriction_C = 0$ and hence $||D^*\varphi||_C^2 = 0$, and so $\varphi$ is constant on $C$. 
Therefore,
the component containing \(\tau\) is the component
containing \(a\) and \(b\). On this component, the maximum occurs at \(a\) and
the minimum occurs at \(b\). Hence
\[
\varphi(a)\geq \varphi(v)>c>\varphi(u)\geq \varphi(b),
\]
so
\[
a\in A,
\qquad
b\notin A.
\]

Let \(\chi_A\in C_0(G')\) be the characteristic function of \(A\). Since
$
Dx=a-b
$,
we get
\[
\langle Dx,\chi_A\rangle
=
\langle a-b,\chi_A\rangle
=
1.
\]
By adjointness,
\[
\langle Dx,\chi_A\rangle
=
\langle x,D^*\chi_A\rangle.
\]
Thus
$
\langle x,D^*\chi_A\rangle=1
$.

For an oriented edge \(e:p\to q\),
\[
\langle D^*\chi_A,e\rangle
=
\chi_A(q)-\chi_A(p).
\]
This is nonzero exactly when \(e\) crosses the cut \(A\mid A^c\).

Moreover, because \(A\) is a strict upper level set of \(\varphi\), every edge
crossing the cut has its \(\varphi\)-value larger on the \(A\)-side and smaller
on the \(A^c\)-side. Therefore each nonzero summand in
\[
\langle x,D^*\chi_A\rangle
=
\sum_e \langle x,e\rangle\langle D^*\chi_A,e\rangle
\]
is positive, and is equal to
$
|\langle x,e\rangle|
$.

Consequently,
\[
1
=
\sum_{e\text{ crosses }A\mid A^c}
|\langle x,e\rangle|.
\]

Since \(\tau\) is one of the edges crossing this cut, we obtain
\begin{equation}
\label{eqn:proof:lem:graph:1}
|\langle x,\tau\rangle|\leq 1.
\end{equation}

If equality held, then \(\tau\) would be the only edge crossing the cut
\(A\mid A^c\). 
If $e = (p,q)$ was another edge crossing the cut, then $\varphi(q) > c \geq \varphi(p)$, and hence $\la x, e \ra > 0$. So for every crossing edge $|\la x, e \ra| > 0$. This together with Eqn. \eqref{eqn:proof:lem:graph:1}, 
and the assumption $|\langle x,\tau\rangle|= 1$ imply that $\tau$ is the only crossing edge. But then deleting \(\tau\) will disconnect \(a\) from
\(b\), and so \(\tau\) will be an \(a\)-\(b\) bridge.

Therefore, if \(\tau\) is not an \(a\)-\(b\) bridge, then
\[
|\langle x,\tau\rangle|<1.
\]

Finally, if \(\tau\) lies in a cycle of \(G'\), then \(\tau\) is not a bridge.
Hence
$
|\langle x,\tau\rangle|<1
$.
\end{proof}

\begin{lemma}[Cycle coefficient implies non-bridge]
\label{lem:bridge}
Let \(G\) be a finite graph, and let \(c\in Z_1(G)\). If
\[
\langle c,\tau\rangle\neq 0
\]
for an edge \(\tau\), then \(\tau\) is not a bridge of \(G\).
\end{lemma}

\begin{proof}
Suppose \(\tau\) were a bridge. Removing \(\tau\) separates the vertices of
\(G\) into two nonempty subsets \(A\) and \(A^c\), with \(\tau\) the only edge
crossing the cut.

Let \(\chi_A\in C_0(G)\) be the characteristic function of \(A\). Since
$
c\in Z_1(G)
$,
we have
$
\partial c=0
$.
Therefore
\[
0
=
\langle \partial c,\chi_A\rangle
=
\langle c,\partial^*\chi_A\rangle.
\]
But \(\partial^*\chi_A\) is supported only on edges crossing the cut, and
\(\tau\) is the only such edge. Thus
\[
\langle c,\partial^*\chi_A\rangle
=
\pm \langle c,\tau\rangle.
\]
Hence
$
\langle c,\tau\rangle=0,
$
contradicting the assumption. Therefore \(\tau\) is not a bridge.
\end{proof}

\begin{proof}[Proof of Theorem~\ref{thm:simplified}]
Fix
\[
\sigma\in S_E(h).
\]
By the characterization of relative essentiality (Lemma~\ref{lem:linear}),
\[
\langle h,\sigma\rangle\neq 0
\qquad\text{and}\qquad
\sigma\perp W_1.
\]

Set
\[
h_\sigma:=\langle h,\sigma\rangle,
y:=\frac{h}{h_\sigma}.
\]
Then
$
\langle y,\sigma\rangle=1
$.

Let
$
\ell_\sigma:Z_1(K)\to\R
$
be the linear functional defined by 
\[
\ell_\sigma(z)=\langle z,\sigma\rangle.
\]

Since
$
\sigma\perp W_1
$,
the functional \(\ell_\sigma\) vanishes on \(W_1\). Hence it descends to a
linear functional
$
\overline{\ell}_\sigma:
Z_1(K)/W_1
\to
\R
$.

Moreover,
\[
\overline{\ell}_\sigma([h])
=
\langle h,\sigma\rangle
=
h_\sigma
\neq 0.
\]

It follows from the hypothesis \eqref{eqn:thm:simplified} that
$
\dim Z_1(K)/W_1=1
$,
and hence 
the functional \(\overline{\ell}_\sigma\) is injective. Therefore
$
\ker\ell_\sigma=W_1.
$

Equivalently,
\[
W_1
=
\{z\in Z_1(K):\langle z,\sigma\rangle=0\}.
\]

Define the affine space
\[
\mathcal A_\sigma
=
\{z\in Z_1(K):\langle z,\sigma\rangle=1\}.
\]

Since
$
\langle y,\sigma\rangle=1
$,
we have
$
y\in\mathcal A_\sigma.
$

If
$
z\in\mathcal A_\sigma
$,
then
$
\langle z-y,\sigma\rangle=0
$
and
$
z-y\in Z_1(K)
$.
Therefore
$
z-y\in W_1
$.
Hence
$
\mathcal A_\sigma=y+W_1
$.

Using Lemma~\ref{lem:orthogonality},
$
h\perp W_1
$,
which implies that 
$
y\perp W_1
$.
Therefore \(y\) is the unique minimum-norm element of \(\mathcal A_\sigma\):
\[
y
=
\argmin
\left\{
\|z\|^2:
z\in Z_1(K),\ \langle z,\sigma\rangle=1
\right\}.
\]

Let
$
G=K^{(1)}
$,
and let
$
G'=G\setminus\{\sigma\}
$.

Let
$
\partial\sigma=b-a
$.

For any
$
z\in\mathcal A_\sigma
$,
the chain
$
z-\sigma
$
is supported on \(G'\), and
\[
\partial(z-\sigma)
=
-\partial\sigma
=
a-b.
\]

Since the coefficient of \(\sigma\) is fixed to be \(1\), minimizing
$
\|z\|^2
$
over \(z\in\mathcal A_\sigma\) is equivalent to minimizing
$
\|z-\sigma\|^2
$
among chains on \(G'\) with boundary \(a-b\).

Thus
$
x:=y-\sigma
$
is the unique minimum-norm chain in \(C_1(G')\) satisfying
$
\partial x=a-b
$.

We now break the proof into two cases.
\begin{enumerate}
\item Case $\tau\in S_N(h)$.

There are now two sub-cases.
\begin{enumerate}
\item
Suppose
\[
\langle h,\tau\rangle=0.
\]
Then
\[
|\langle h,\tau\rangle|
=
0
<
|\langle h,\sigma\rangle|,
\]
because \(\sigma\in S_E(h)\) implies
$
\langle h,\sigma\rangle\neq 0
$.

\item
Now suppose 
$
\langle h,\tau\rangle\neq 0.
$

Since
$
\tau\in S_N(h)
$,
the edge \(\tau\) is not relatively essential. Because its \(h\)-coefficient is
nonzero, 
Lemma~\ref{lem:linear}
implies that
$
\tau\not\perp W_1
$.
Thus there exists
$
w\in W_1
$
such that
$
\langle w,\tau\rangle\neq 0
$.

Since
$
\sigma\in S_E(h),
$
we have
$
\sigma\perp W_1.
$
Therefore
$
\langle w,\sigma\rangle=0.
$

Also,
$
w\in W_1\subseteq Z_1(K)
$,
and 
so \(w\) is a cycle in \(K\). Since its \(\sigma\)-coefficient is zero, \(w\)
is a cycle in \(G'\). Since
$
\langle w,\tau\rangle\neq 0
$,
Lemma~\ref{lem:bridge}
implies that \(\tau\) is not a bridge of \(G'\).

By 
Lemma~\ref{lem:graph}
applied to
$
x=y-\sigma
$,
we get
$
|\langle x,\tau\rangle|<1
$.

Since \(\tau\neq\sigma\), and \(x=y-\sigma\), we have
\[
\langle x,\tau\rangle
=
\langle y,\tau\rangle.
\]
Hence
$
|\langle y,\tau\rangle|<1
$.

Finally,
$
y=\frac{h}{\langle h,\sigma\rangle}
$, and
hence
\[
\langle y,\tau\rangle
=
\frac{\langle h,\tau\rangle}{\langle h,\sigma\rangle}.
\]
Thus
\[
\left|
\frac{\langle h,\tau\rangle}
     {\langle h,\sigma\rangle}
\right|
<1.
\]
Equivalently,
\[
|\langle h,\tau\rangle|
<
|\langle h,\sigma\rangle|.
\]
\end{enumerate}
The two cases exhaust all
$
\tau\in S_N(h)
$.

\item Case $\tau \in S_E(h)$. 
Suppose that $z \in Z_1(K)$. We claim that 
\begin{equation}
\label{eqn:proof:thm:simplified:0}
  \la z, \sigma\ra \neq 0  \Rightarrow \la z, \tau\ra \neq 0
\end{equation}
Indeed $\tau \perp W_1$ and $\la h, \tau \ra \neq 0$. Moreover,  
$Z_1(K) = \operatorname{span}(h) \oplus W_1$, and hence
$z = c h + w$ for some $c \in \mathbb{R}, w \in W_1$. Since $\sigma \perp W_1$,
$c \neq 0$. This implies that 
$\la z , \tau\ra = c \cdot \la h,\tau\ra \neq 0$ as well. 

Since $\operatorname{ker}(\ell_\sigma) = \operatorname{ker}(\ell_\tau) = W_1$,
and $W_1$ is of codimension $1$ in $Z_1(K)$, the linear functional 
\begin{equation}
\label{eqn:proof:thm:simplified}
\ell_\sigma = \lambda \ell_\tau,
\end{equation}
for some $\lambda \neq 0$. 
Since $\la h, \sigma \ra \neq 0$ and $h \in Z_1(K)$, Lemma~\ref{lem:bridge} implies that $\sigma$ is not a bridge in $G$. So there exists a simple cycle $z_1$
(i.e.
$z_1 \in Z_1(K)$ with all coefficients in $\{0,1,-1\}$), with $\sigma$ in its support. Then 
$\ell_\sigma(z_1) = \la z_1,\sigma\ra = \pm 1$, and hence $\ell_\tau(z_1) = \la z_1,\tau\ra  \neq 0$ using \eqref{eqn:proof:thm:simplified:0}. Therefore,
$\ell_\tau(z_1) \in \{1,-1\}$. 
Comparing with \eqref{eqn:proof:thm:simplified} we get that $\lambda \in \{1,-1\}$.
Evaluating $\ell_\sigma$ and $\ell_\tau$ at $h$, we obtain
$\la h, \tau\ra = \pm \la h, \sigma\ra$.
\end{enumerate}
Therefore the theorem is proved.
\end{proof}

\begin{proof}[Proof of Theorem~\ref{thm:main}]
There are two cases to consider.
\begin{enumerate}
\item
Case $p=0$. 
To see this suppose that 
$b = (s,t;1) \in \mathbf{B}_0(\mathcal{F})$ is a generic bar,
and $h = \sum_{\sigma \in K^{[0]}} c_\sigma \cdot \sigma$ is 
a harmonic representative of $\mathcal{P}^{s,t}_0$. 
\hide{
Each $0$-dimensional simplex (i.e. vertex) $\sigma$ belonging to  $\supp(h)$,
belongs to a connected component of $K_s$ disjoint from $K_{s-1}$.
To see this observe that $h \perp W_0 \supset Z_0(K_{s-1}) = C_0(K_{s-1})$
using Lemma~\ref{lem:orthogonality}. So $\supp(h) \cap K_{s-1}^{[0]} = \emptyset$. 
}

First note that $h \perp W_0$ (Lemma~\ref{lem:orthogonality}).
Each $0$-dimensional simplex (i.e. vertex) $\sigma$ belonging to  $\supp(h)$,
has the property that the corresponding $0$-cycle $\sigma \in Z_0(K_s)$ belongs to $\Rep(b)$.

To see this first observe that
$\la h, \sigma \ra \neq 0$ and $h \perp W_0$ together gives $\sigma \not\in W_0$.

Moreover, if $t \neq \infty$, $ \widetilde{M}_0^{s,t} = Z_0(K_s)$, and $\widetilde{M}_0^{s,t-1} = W_0$, and if $t = \infty$, $\widetilde{M}_0^{s,N} = W_0$ (see Lemma~\ref{lem:h}). 
In either case $\sigma \in \Rep(b) = Z_0(K_s) \setminus W_0$.

Thus, if $\supp(h)$ has more than one element, then the set $\Sigma(b) = \bigcap_{z \in \Rep(b)} \supp(z) \subset \bigcap_{\sigma \in \supp(h)} \{\sigma\}$ is empty, and there is nothing to prove. Otherwise, $\supp(h)$ has exactly one element,  and in this case the theorem is obviously true.
\item
Case $p=1$.
It follows from Theorem~\ref{thm:simplified}, 
 Proposition~\ref{prop:simplified} and Lemma~\ref{lem:simplified}, that
 for each $\sigma \in \Sigma(b) = S_E(h)$, and $\tau \in K_s^{[1]}, \tau \neq \sigma$,
 \begin{equation}
 \label{eqn:proof:thm:main}
   |\la h, \sigma\ra| \geq  |\la h, \tau \ra|  
 \end{equation} 
 with equality if and only if $\tau \in S_E(h)$. Notice that for $\tau \in K^{[1]} \setminus K_s^{[1]}$, $|\la h, \tau \ra| = 0$ and so inequality~\eqref{eqn:proof:thm:main} holds for all such $\tau$.  
\end{enumerate}
This completes the proof of Theorem~\ref{thm:main}.
\end{proof}

\section{Constructing a counterexample to the coefficient dominance conjecture in the general case.}
\label{sec:counterexample}
\subsection{Counter-example for $p=2$}
\label{subsec:counter:2}
The following construction provides a simplicial complex $J$ that contains the real projective plane $\mathbb{R}P^2$ as a subcomplex and 
has the same homology groups (over $\mathbb{Z}$) as the $2$-sphere $\mathbb{S}^2$.

\subsubsection{Triangulation of $\mathbb{R}P^2$}
We begin with a triangulation of the real projective plane $\mathbb{R}P^2$:

Let the vertex set be $V_0 = \{0, 1, 2, 3, 4, 5\}$.
The facets of $\mathbb{R}P^2$ are
   \[
   [0, 1, 2], [0, 2, 3], [0, 3, 4], [0, 4, 5], [0, 1, 5], [1, 2, 4], [1, 3, 4], [1, 3, 5], [2, 3, 5], [2, 4, 5].
   \]

\subsubsection{Möbius Strip Subcomplex}

The complex $\mathbb{R}P^2$ can be decomposed into a Möbius strip $M$ and a 
$2$-disk (the ``cap''). The Möbius strip subcomplex $M \subset \mathbb{R}P^2$ is formed by the facets that do not contain the vertex $0$. The facets of $M$ are
    $$[1, 2, 4], [1, 3, 4], [1, 3, 5], [2, 3, 5], [2, 4, 5].$$

The first homology group $\HH_1(M;\mathbb{Z}) \cong \mathbb{Z}$ is generated by the cycle $c = [1,3] + [3,5] - [2,5] + [2,4] - [1,4]$ crossing through the Möbius strip.

In $\mathbb{R}P^2$, the homology class $[c]$ generates the torsion subgroup $H_1(\mathbb{R}P^2;\mathbb{Z}) \cong \mathbb{Z}/2\mathbb{Z}$.

\subsubsection{Attaching a Disk to Kill Torsion}
To obtain a simplicial complex homologically equivalent to $\mathbb{S}^2$, we must kill the torsion in $\HH_1(\mathbb{R}P^2;\mathbb{Z})$. This is achieved by attaching a $2$-disk along the core cycle $c$.

\noindent We introduce a new vertex $6$ and triangulate the disk attached along $c$. Since $c$ is a cycle of length $5$, we require $5$ new triangles connecting vertex $6$ to the edges of $c$.

\noindent The new facets (triangles) are:
$$[1, 3, 6], [3, 5, 6], [2, 5, 6], [2, 4, 6], [1, 4, 6].$$

\subsubsection{The simplicial complex $J$}
The facets of $J$ are
$$[0, 1, 2], [0, 1, 5], [0, 2, 3], [0, 3, 4], [0, 4, 5], [1, 2, 4], [1, 3, 4],$$
$$[1, 3, 5], [1, 3, 6], [1, 4, 6], [2, 3, 5], [2, 4, 5], [2, 4, 6], [2, 5, 6],[3, 5, 6],$$
and note that we have obtained $J$ by only adding simplices to a triangulation of $\mathbb{R}P^2$.
\subsubsection*{$\HH_*(J;\mathbb{Z}) \cong \HH_*(\mathbb{S}^2;\mathbb{Z})$}
\begin{itemize}
    \item \textbf{$\HH_0(J;\mathbb{Z}) \cong \mathbb{Z}$:} $J$ is path-connected by construction.
    \item \textbf{$\HH_1(J;\mathbb{Z}) = 0$:} The torsion generator of $\mathbb{R}P^2$ was the core cycle $c$. By attaching the disk along $c$, we make $c$ a boundary, killing the $\mathbb{Z}/2\mathbb{Z}$ torsion subgroup. Thus $\HH_1(J;\mathbb{Z}) = 0$.
    \item \textbf{$\HH_2(J;\mathbb{Z}) \cong \mathbb{Z}$:} This follows directly from Euler characteristic being $2$ (by counting simplices). 
    Since $\chi(J) = 2$ and $\HH_1(J;\mathbb{Z}) = 0$, we must have $\HH_2(J;\mathbb{Z}) \cong \mathbb{Z}$.
\end{itemize}

Consequently, $J$ has the same homology groups as $\mathbb{S}^2$ while containing $\mathbb{R}P^2$ as a subcomplex.

Alternatively, we can check using a direct calculation that
\[
\card(J^{[0]}) =7,\qquad \card(J^{[1]}) =20,\qquad \card(J^{[2]})=15,
\]
and,
\[
\operatorname{rank}\partial_1=6,
\qquad
\operatorname{rank}\partial_2=14.
\]

Thus,

\[
\dim \HH_1(J)=20-6-14=0,
\]

and

\[
\dim \HH_2(J)=15-14=1.
\]

We will now add a tetrahedron, $[2,4,5,6]$, to the complex $J$ so that the boundary subspace (through Hodge decomposition) of $C_2$ of the resulting complex is non-trivial. Let us call the complex obtained by adding such simplex and all its missing facets as $\widetilde{J}_2$. The tetrahedron chosen is such that we do not introduce a new generator in homology (effectively just slightly thickening a section of our complex). 

We have
\[
\card(\widetilde{J}_2^{[0]})=7,\card(\widetilde{J}_2^{[1]})=20, \card(\widetilde{J}_2^{[2]})=16, \card(\widetilde{J}_2^{[3]})=1,
\]
and
\[
\operatorname{rank}\partial_2=14,
\operatorname{rank}\partial_3=1.
\]

Then 
\[
\dim \HH_2(\widetilde{J}_2)=(16-14)-1=1.
\]

An elementary computation yields a generator 
\[
\begin{split}
h= - 4 [0, 1, 2]+ 4 [0, 1, 5]- 4 [0, 2, 3]- 4 [0, 3, 4]- 4 [0, 4, 5]+ 4 [1, 2, 4]+ 4 [1, 3, 4]+ 4 [1, 3, 5]\\- 8 [1, 3, 6]+ 8 [1, 4, 6]+ 4 [2, 3, 5]- 1 [2, 4, 5]- 3 [2, 4, 6]+ 3 [2, 5, 6]- 8 [3, 5, 6]+ 5 [4, 5, 6].
\end{split}
\]

Note that
$
\partial_2 h=0
$,
and
$
h\perp B_2(\widetilde{J}_2)=\operatorname{Im}\partial_3
$.

Indeed, with

\[
T=[2,4,5,6],
\]

one has

\[
\partial T=[4,5,6]-[2,5,6]+[2,4,6]-[2,4,5].
\]

Using the coefficients of $h$,
\[
\langle h,\partial T\rangle= 5-3-3+1 = 0.
\]

Note that the projection of the simplex $\tau = [4,5,6]$ onto 
\[
B_2(\widetilde{J}_2)=\operatorname{Im}\partial_3(\widetilde{J}_2)\subset C_2(\widetilde{J}_2)
\]
is non-trivial (that is, $\tau \not\perp B$). We also have that $|\langle h, \tau\rangle| = 5 $, while for other simplices that are orthogonal to $B$ such as $\sigma = [0,1,2]$ we have  $|\langle h, \sigma \rangle| = 4$, a smaller coefficient than that of $\tau$. 
\hide{
Thus we have shown that the following is true:

\begin{theorem}\label{thm:base_case}
    There exists a complex $K_2$ of dimension $3$ such that
    \begin{enumerate}[label=(\alph*)]
        \item $\dim \HH_2(K_2) = 1$,
        \item For $h$ a harmonic cycle, there exist simplices $\sigma\perp B$ and $\tau\not\perp B$, where $B$ is the space of $2$-dimensional boundaries, such that $|\langle h, \sigma \rangle| < |\langle h,\tau \rangle|$.
    \end{enumerate}
\end{theorem}
Taking the two-step filtration \(\varnothing\subset K_2\), the unique nonzero
class in \(\HH_2(K_2)\) gives a simple generic infinite bar. The
displayed harmonic cycle is a harmonic representative of this bar.
}

We have the following theorem.
\begin{theorem}\label{thm:base_case}
    There exists a complex $\widetilde{J}_2$ of dimension $3$ such that
    \begin{enumerate}[label=(\alph*)]
        \item $\dim \HH_2(\widetilde{J}_2) = 1$,
        \item For the filtration $\FF$, $\emptyset = K_0  \subset K_1 = \widetilde{J}_2$, there exists a generic bar $b=(1,\infty;1) \in \mathbf{B}_2(\FF)$, with a harmonic representative $h$, and $\sigma \in \Sigma(b), \tau \in \widetilde{J}_2^{[2]} \setminus \Sigma(b)$ 
        such that $|\langle h, \sigma \rangle| < |\langle h,\tau \rangle|$.
    \end{enumerate}
\end{theorem}
\begin{proof}
The unique nonzero
class in \(\HH_2(\widetilde{J}_2)\) gives a simple generic infinite bar. The
displayed harmonic cycle $h$ is a harmonic representative of this bar, since
$h \perp B_2(\widetilde{J}_2) =  W_2$ (recall definition of $W_2$ from Eqn. \eqref{eqn:W}).
Moreover, using Lemma~\ref{lem:linear},
$\sigma \in S_E(h)$, $\sigma \perp W_2 = B_2(\widetilde{J}_2)$ and $\la h, \sigma\ra \neq 0$, but $\tau \not\in S_E(h)$ since $\tau \not\perp W_2$.
But as shown above, 
\[
4 = |\langle h, \sigma \rangle| < |\langle h,\tau \rangle| = 5.
\]
Finally, using Lemma~\ref{lem:simplified}, we have $S_E(h) = \Sigma(b)$.
\end{proof}

\subsection{Counter-example for $p> 2$}
\label{subsec:counter:p}
\begin{definition}
    Let $A, B$ be simplicial complexes with disjoint set of vertices. The join $A*B$ is defined as the simplicial complex with simplices $\{a\cup b \mid a\in A, b\in B\}$. 
\end{definition}
Let $n\geq 3$. Let $L$ be a triangulation of $\mathbb{S}^{n-3}$. 

\begin{theorem}
\label{thm:counter-example}
    There exists a complex $\widetilde{J}_n$ of dimension $n+1$ such that
    \begin{enumerate}
        \item $\dim \HH_n(\widetilde{J}_n) = 1$,
        
        \item For the filtration $\FF$, $\emptyset = K_0  \subset K_1 = \widetilde{J}_n$, there exists a generic bar $b=(1,\infty;1) \in \mathbf{B}_n(\FF)$, with a harmonic representative $h$, and $\sigma \in \Sigma(b), \tau \in \widetilde{J}^{[n]}_n \setminus \Sigma(b)$ 
        such that $|\langle h, \sigma \rangle| < |\langle h,\tau \rangle|$.
        \hide{
        For a harmonic representative $h$,  there exist simplices $\sigma\perp B$ and $\tau\not\perp B$, where $B$ is the space of $n$-dimensional boundaries, such that $|\langle h, \sigma \rangle| < |\langle h,\tau \rangle|$.
        }
    \end{enumerate}
\end{theorem}
\begin{proof}
    Let us define $\widetilde{J}_n = \widetilde{J}_2* L$. 
    The Künneth formula for joins states that
    \[
\widetilde \HH_m(A*B)
\cong
\bigoplus_{r+s=m-1}
\widetilde \HH_r(A)\otimes
\widetilde \HH_s(B)
\]
(where $\widetilde \HH_*(\cdot)$ denotes reduced homology).
Applying it with

\[
A=\widetilde{J}_2,\qquad B=L\simeq S^{n-3}.
\]
we obtain that 
$\dim \HH_n(\widetilde{J}_n) =1 $.

    Let $h'$ be the harmonic cycle in Theorem \ref{thm:base_case}. Let $\omega$ be the fundamental class of $L$. Define 
    \[
    h = \sum_{\substack{
    \alpha\in \widetilde{J}_2^{[2]}\\ \beta\in L^{[n-3]}
    }} \langle h',\alpha \rangle \langle \omega,\beta \rangle \alpha * \beta.
    \]
And note that for all simplices $\eta = \alpha * \beta$ we have $\langle h, \eta\rangle =  \langle h',\alpha \rangle \langle \omega,\beta \rangle$.
To see that $h$ is a harmonic cycle in $\widetilde{J}_n$ we have to show that it is both a cycle and a cocycle.
\begin{itemize}
    \item $h$ is a cycle:
    Recall that the action of the boundary on a simplex of the join is given by
    \[
    \partial (\alpha * \beta)=(\partial\alpha)*\beta +(-1)^{\dim(\alpha) +1}\alpha*(\partial \beta).
    \]
    Thus
    \begin{align*}
    \partial h &= \partial \sum_{\alpha,\beta}  \langle h',\alpha \rangle \langle \omega,\beta \rangle \alpha * \beta\\
    &= \sum_{\alpha,\beta}  \langle h',\alpha \rangle \langle \omega,\beta \rangle \partial ( \alpha * \beta)\\
    &= \sum_{\alpha,\beta}  \langle h',\alpha \rangle \langle \omega,\beta \rangle \left((\partial\alpha)*\beta +(-1)^{\dim(\alpha) +1}\alpha*(\partial \beta)\right)\\
    &= \sum_{\alpha,\beta}  \langle h',\alpha \rangle \langle \omega,\beta \rangle \left((\partial\alpha)*\beta -\alpha*(\partial \beta)\right)\\
     &= \sum_{\alpha,\beta}  \langle h',\alpha \rangle (\partial\alpha) * \langle \omega,\beta \rangle \beta - \langle h',\alpha \rangle \alpha*\langle \omega,\beta \rangle (\partial \beta)\\
     &=\partial h'* \omega - h' * \partial \omega = 0
\end{align*}
    \item $h$ is a cocycle: 
    By definition, 
    $dh=0$ if $\langle dh, \eta\rangle=0$  for all $n+1$ (top dimensional) simplices $\eta$, where $d = \partial^*$. 
In other words,
\[
h\perp B_n(\widetilde{J}_n),
\]
i.e.
\[
\langle h,\partial \eta\rangle=0
\]
for every \((n+1)\)-simplex \(\eta\).
    By construction, the top simplices of $\widetilde{J}_n$ are joins of the only tetrahedron $T$ of $\widetilde{J}_2$ with top simplices of $L$. For a given $\beta$ top dimensional, their boundary evaluates to
    \[
    \partial (T*\beta) = \partial(T)*\beta + T*\partial \beta
    \] 
    We have 
    \begin{align*}
        \langle dh, \eta\rangle &= \langle dh, T*\beta\rangle = \langle h, \partial(T*\beta)\rangle = \langle h,  \partial(T)*\beta + T*\partial \beta \rangle 
    \end{align*}
    By dimensionality, $\langle h, T*\partial \beta \rangle$ is zero, therefore  
    \begin{align*}
        \langle dh, \eta\rangle &= \langle h,  \partial(T)*\beta \rangle \\
        &=\left\langle \sum_{\alpha,\beta'}  \langle h',\alpha \rangle \langle \omega,\beta' \rangle \alpha * \beta',\,  \partial(T)*\beta \right\rangle\\
        &= \sum_{\alpha}  \langle h',\alpha \rangle \langle \omega,\beta \rangle \left\langle\alpha * \beta,\,  \partial(T)*\beta \right\rangle\\
        &= \langle \omega,\beta \rangle  \sum_{\alpha\subseteq \partial T}  \langle h',\alpha \rangle \left\langle \alpha * \beta,\,  \partial(T)*\beta \right\rangle
    \end{align*}

    Since $\partial T = \sum_{\alpha'\subseteq \partial T} \langle \partial T, \alpha'\rangle \alpha'$
\begin{align*}
        \langle dh, \eta\rangle &= \langle \omega,\beta \rangle  \sum_{\alpha\subseteq \partial T}  \langle h',\alpha \rangle \left\langle \alpha * \beta,\,  \sum_{\alpha'\subseteq \partial T} \langle \partial T, \alpha'\rangle \alpha'*\beta \right\rangle\\
        &= \langle \omega,\beta \rangle  \sum_{\alpha\subseteq \partial T}  \langle h',\alpha \rangle \langle \partial T, \alpha \rangle =0.
\end{align*}
\end{itemize}

Now consider the simplices $\sigma = [0,1,2]$ and $\tau = [4,5,6]$ in $\widetilde{J}_2$, and a top-dimensional simplex $\beta_0$ of $L$, noting that $\la \omega,\beta_0\ra = \pm 1$. We have that $\tau * \beta_0$ is not orthogonal to boundaries and $\sigma *\beta_0$ is, as well as 
\[
|\langle h, \tau*\beta_0 \rangle| = 5 > 4 = |\langle h, \sigma*\beta_0 \rangle|.
\]
The rest of the proof is identical to that of Theorem~\ref{thm:base_case}.
\end{proof}

\section{Genericity is required}
\label{sec:appendix}
Recall that the genericity property has two parts which we recall below:

\begin{itemize}
\item Property \eqref{itemlabel:generic:1}: the bar is simple, i.e.
\[
\mu_p^{s,t}(F)=1;
\]
\item Property \eqref{itemlabel:generic:2}: 
no other bar has the same birth time \(s\).
\end{itemize}

\subsection{If Property \eqref{itemlabel:generic:1} does not hold}
The simplicity assumption is necessary for the statement to be non-vacuous. Indeed, if $b=(s,t;m)$ has multiplicity $m>1$, then, under the natural extension of the definition of representatives, 
$\Sigma(b) = \emptyset$.
Moreover, if one tries to choose an arbitrary harmonic representative inside the $m$-dimensional harmonic bar subspace and define essential simplices relative to that choice, coefficient dominance can fail.

\hide{

The coefficient-dominance theorem is stated for simple bars, because only then
the harmonic barcode subspace attached to the bar is one-dimensional and hence
is spanned by a single harmonic representative, unique up to nonzero scalar.

If Property \eqref{itemlabel:generic:1} is omitted, then the theorem is not literally well-defined in
the same form: a bar of multiplicity \(m>1\) has a harmonic barcode subspace of
dimension \(m\), not a distinguished harmonic line.

There are two useful facts.

First, under the natural extension of the definition of representatives to a
bar of multiplicity \(m>1\), the set of essential simplices is empty. Thus the
coefficient-dominance statement becomes vacuous.

Second, if one tries to choose an arbitrary nonzero harmonic vector in the
higher-dimensional harmonic bar subspace and define essential simplices relative
to that chosen vector, then coefficient dominance can fail.
}
We explain both points below.
\subsubsection{Multiplicity greater than one forces the essential set to be empty}

Let \(b=(s,t;m)\) be a bar of multiplicity
$m>1$.
\hide{
Let \(M\) be the corresponding representative vector space and let
\(N\subset M\) be the previous, or null, subspace. Thus
\[
\dim(M/N)=m>1.
\]
}

If $t \neq \infty$, let
\[
M=\widetilde M_p^{s,t}(F),
\qquad
N=\widetilde M_p^{s,t-1}(F),
\]
and 
if $t = \infty$ let
\[
M=Z_p(K_s),
\qquad
N=\widetilde M_p^{s,N}(F).
\]

Then,
$
\Rep(b)=M\setminus N.
$
(see Definition~\ref{def:rep}).

\begin{lemma}
If \(\dim(M/N)>1\), then
\[
\Sigma(b)=\varnothing.
\]
\end{lemma}

\begin{proof}
Let \(\sigma\) be any \(p\)-simplex. Define the coordinate functional
\[
\ell_\sigma:M\to \mathbb{R},
\ell_\sigma(z)=\langle z,\sigma\rangle.
\]

If
$
\ell_\sigma\equiv 0
\quad\text{on }M,
$
then every representative has zero \(\sigma\)-coefficient, so \(\sigma\notin
\Sigma(b)\).

Now suppose that
$\ell_\sigma$
is not identically zero on \(M\). 
Then
$
\ker(\ell_\sigma\restriction_M)
$
is a hyperplane in \(M\), so
$
\dim\ker(\ell_\sigma|_M)=\dim M-1.
$
Since
$
\dim(M/N)=m>1,
$
we have
$
\dim N=\dim M-m\leq \dim M-2
$.
Therefore
$
\ker(\ell_\sigma|_M)\not\subseteq N
$.
Hence there exists
$
z\in \ker(\ell_\sigma|_M)\setminus N
$.

This \(z\) lies in
$
M\setminus N=\Rep(b)
$,
but
$
\langle z,\sigma\rangle=0
$.
Thus
$
\sigma\notin\Sigma(b)
$.

Since \(\sigma\) was arbitrary,
\[
\Sigma(b)=\varnothing.
\]
\end{proof}

\begin{remark}
This shows that, without Property \eqref{itemlabel:generic:1}, the usual essential-simplex definition
does not give a nontrivial coefficient-dominance statement. For bars of
multiplicity greater than one, the intersection of supports over all
representatives is empty.
\end{remark}

\subsubsection{A concrete non-simple bar}

We now give a small explicit example showing that a naive \(h\)-level extension
of coefficient dominance fails.

Let \(K_0\) consist of the six vertices
\[
0,1,2,3,4,5.
\]

Let \(K_1=K\) have edges
\[
[01],[02],[12],[03],[13],[04],[45],[05],
\]
and one \(2\)-simplex
\[
[012].
\]

Orient every edge by increasing vertex order.

Consider the two-step filtration
\[
K_0\subset K_1=K.
\]

The graph \(K^{(1)}\) is connected with
$6$
vertices and $8$
edges. Hence its cycle rank is
$
8-6+1=3.
$

The single triangle contributes one independent boundary,
\[
\partial[012]=[12]-[02]+[01].
\]
Therefore
\[
\dim \HH_1(K)=3-1=2.
\]

Since
$
\HH_1(K_0)=0
$,
the barcode in degree \(1\) has one infinite bar of multiplicity \(2\):
$
b=(1,\infty;2).
$

This bar satisfies the analogue of Property \eqref{itemlabel:generic:2}, because no other bar has birth
time \(1\). But it fails Property \eqref{itemlabel:generic:1}, because
$
\mu_1^{1,\infty}(F)=2.
$

\subsubsection{The harmonic subspace}

Define
\[
h_1=-2[01]-[02]+[12]+3[03]-3[13],
\]
and
\[
h_2=[04]+[45]-[05].
\]

We check that both are harmonic \(1\)-cycles in \(K\).

First,
\[
\partial h_1=0.
\]
Indeed, using \(\partial[ij]=[j]-[i]\), the vertex coefficients cancel:
\[
\begin{aligned}
\partial h_1
={}&-2([1]-[0])-([2]-[0])+([2]-[1])\\
&+3([3]-[0])-3([3]-[1])\\
={}&0.
\end{aligned}
\]

Also,
\[
\partial h_2
=
([4]-[0])+([5]-[4])-([5]-[0])
=
0.
\]

Thus
\[
h_1,h_2\in Z_1(K).
\]

Next,
\[
B_1(K)=\mathrm{span}\{\partial[012]\}.
\]
Since
\[
\partial[012]=[12]-[02]+[01],
\]
we have
\[
\langle h_1,\partial[012]\rangle
=
1-(-1)-2
=
0.
\]
Also,
\[
\langle h_2,\partial[012]\rangle=0,
\]
because \(h_2\) is supported on the edges
\[
[04],[45],[05],
\]
which do not occur in \(\partial[012]\).
Therefore
$
h_1,h_2\perp B_1(K)
$,
and so
$
h_1,h_2\in \mathcal{H}_1(K)
$.

Since $
\dim H_1(K)=2
$,
the harmonic bar subspace attached to
$
b=(1,\infty;2)
$
is equal to
$
\mathrm{span}\{h_1,h_2\}
$.

\subsubsection{Failure of a naive \(h\)-level extension}
Although the bar itself has no essential simplices under the natural definition,
one might try the following alternative construction. Choose a nonzero
harmonic vector
$
h\in\mathrm{span}\{h_1,h_2\}
$
and define essential simplices relative to the chosen vector \(h\), rather than
relative to the full multiplicity-two bar.

We show that coefficient dominance fails for this naive extension.

Let
$
0<\varepsilon<2
$,
and set
$
h=h_1+\varepsilon h_2
$.
Explicitly,
\[
h
=
-2[01]-[02]+[12]+3[03]-3[13]
+\varepsilon[04]+\varepsilon[45]-\varepsilon[05],
\]

and
\[
B_1(K)=\mathrm{span}\{[12]-[02]+[01]\}.
\]
Consider the edge
$
\sigma=[04]
$.
This edge does not occur in \(\partial[012]\), so
$
\sigma\perp B_1(K)
$.

Also,
\[
\langle h,\sigma\rangle
=
\langle h,[04]\rangle
=
\varepsilon
\neq 0.
\]
Thus \([04]\) is essential relative to the chosen harmonic vector \(h\).

Now consider the edge
$
\tau=[01]
$.
Since
$
\langle [01],\partial[012]\rangle=1
$,
we have
$
\tau\not\perp B_1(K)
$.
Thus \([01]\) is non-essential relative to \(h\).

But
$
\langle h,[01]\rangle=-2
$.
Therefore
\[
|\langle h,\sigma\rangle|
=
|\langle h,[04]\rangle|
=
\varepsilon,
\]
whereas
\[
|\langle h,\tau\rangle|
=
|\langle h,[01]\rangle|
=
2.
\]

For any
$
0<\varepsilon<2
$,
we get
$
|\langle h,\tau\rangle|
>
|\langle h,\sigma\rangle|
$.

Thus coefficient dominance fails for this naive choice of a harmonic vector
inside the multiplicity-two harmonic bar subspace.

\hide{
\subsubsection{Conclusion}

Removing Property (A) has two consequences.

First, the coefficient-dominance theorem is no longer literally well-defined
as stated, because the harmonic object attached to a multiplicity-\(m\) bar is
an \(m\)-dimensional subspace, not a harmonic line.

Second, under the natural extension of the representative-set definition, if
\[
m>1,
\]
then
\[
\Sigma(b)=\varnothing.
\]
Thus the coefficient-dominance statement becomes vacuous.

Finally, if one instead chooses an arbitrary nonzero harmonic representative
inside the higher-dimensional harmonic bar subspace and defines essential
simplices relative to that choice, coefficient dominance can fail. The explicit
example above gives a bar
\[
b=(1,\infty,2)
\]
satisfying the analogue of Property (B) but not Property (A), and a chosen
harmonic vector
\[
h=h_1+\varepsilon h_2
\]
for which the essential edge
\[
\sigma=[04]
\]
has coefficient magnitude
\[
\varepsilon,
\]
while the non-essential edge
\[
\tau=[01]
\]
has coefficient magnitude
\[
2.
\]
For \(0<\varepsilon<2\),
\[
|\langle h,\tau\rangle|>|\langle h,\sigma\rangle|.
\]
}
Therefore the simplicity assumption is necessary for the coefficient-dominance
theorem to have a non-vacuous, canonical formulation.

\subsection{If Property \eqref{itemlabel:generic:2} does not hold}
In this case, 
we construct a filtration and a simple one-dimensional bar
\[
b=(1,3;1)
\]
such that another bar has the same birth time. Thus \(b\) satisfies 
Property~\eqref{itemlabel:generic:1} but fails Property~\eqref{itemlabel:generic:2}.
For this non-generic simple bar, we show that the harmonic representative
has a non-essential edge with coefficient larger in absolute value than an
essential edge.


Let
\[
K_0\subset K_1\subset K_2\subset K_3
\]
be the following filtration.

Let
\[
K_0=\{[0],[1],[2],[3],[4]\}
\]
be the complex consisting of five isolated vertices.

Let \(K_1\) be obtained from \(K_0\) by adding following seven edges
\[
[01],\ [02],\ [04],\ [13],\ [14],\ [23],\ [24].
\]

Thus \(K_1\) is a connected graph with
$5$ vertices and $7$ edges.
Therefore
\[
\dim H_1(K_1)=7-5+1=3.
\]

Define two cycles in \(K_1\):
\[
n=[02]+[23]-[13]+[14]-[04],
\]
and
\[
d=[01]+[13]-[23]+[24]-[04].
\]

The cycle \(n\) is the oriented cycle
\[
0\to 2\to 3\to 1\to 4\to 0,
\]
and the cycle \(d\) is the oriented cycle
\[
0\to 1\to 3\to 2\to 4\to 0.
\]

Now define \(K_2\) by attaching a cone over the cycle \(n\), with new vertex
\(5\). Concretely, add the edges
\[
[05],\ [15],\ [25],\ [35],\ [45]
\]
and the triangles
\[
[025],\ [235],\ [135],\ [145],\ [045].
\]

Let
\[
D_n=[025]+[235]-[135]+[145]-[045].
\]
Using
\[
\partial[abc]=[bc]-[ac]+[ab]
\]
for \(a<b<c\), one checks that
\[
\partial D_n
=
[02]+[23]-[13]+[14]-[04]
=
n.
\]
Thus the passage from \(K_1\) to \(K_2\) kills the class represented by \(n\).

Finally, define \(K_3\) by attaching a cone over the cycle \(d\), with new
vertex \(6\). Concretely, add the edges
\[
[06],\ [16],\ [26],\ [36],\ [46]
\]
and the triangles
\[
[016],\ [136],\ [236],\ [246],\ [046].
\]

Let
\[
D_d=[016]+[136]-[236]+[246]-[046].
\]
Then
\[
\partial D_d
=
[01]+[13]-[23]+[24]-[04]
=
d.
\]
Thus the passage from \(K_2\) to \(K_3\) kills the class represented by \(d\),
modulo the class already killed by \(n\).


At time \(1\), three independent one-dimensional homology classes are born,
because
\[
\dim H_1(K_1)=3.
\]

At time \(2\), the class represented by \(n\) dies. At time \(3\), the class
represented by \(d\) dies modulo the class already killed by \(n\). One
remaining independent class persists.

Thus the one-dimensional barcode contains
\[
(1,2;1),
(1,3;1),
(1,\infty;1).
\]

Let $b=(1,3;1)$.
The bar $b$ is simple. However, it is not generic, because the distinct bar
$
(1,2;1)
$
has the same birth time \(1\). Thus \(b\) 
fails 
Property~\eqref{itemlabel:generic:2}.


At time \(1\), there are no \(2\)-simplices, so
\[
B_1(K_1)=0.
\]
Consequently every cycle in \(K_1\) is harmonic as a cycle in \(K_1\).

The class killed at time \(2\) is represented by \(n\). Hence
\[
\widetilde{M}_1^{1,2}=\mathrm{span}\{n\}.
\]

The classes killed by time \(3\) are represented by \(n\) and \(d\). Hence
\[
\widetilde{M}_1^{1,3}=\mathrm{span}\{n,d\}.
\]

For the bar \(b=(1,3;1)\), the representative cycles are the elements of
$
\widetilde{M}_1^{1,3}\setminus \widetilde{M}_1^{1,2}
$.
Equivalently,
\[
\Rep(b)
=
\{a d+b n : a\in\mathbb{R}^\times,\ b\in\mathbb{R}\}.
\]

Indeed, the coefficient of \(d\) must be nonzero in order to represent a
nonzero element of
$
M_1^{1,3}/M_1^{1,2}
$


Since \(K_1\) has no \(2\)-simplices, the harmonic representative of \(b\) is
the vector in
$
\mathrm{span}\{n,d\}
$
orthogonal to
$
\mathrm{span}\{n\}
$.

We have
\[
\langle d,n\rangle=-1,
\|n\|^2=5.
\]
Therefore the vector in \(\mathrm{span}\{n,d\}\) orthogonal to \(n\) is
\[
h
=
d-\frac{\langle d,n\rangle}{\|n\|^2}n
=
d+\frac15 n.
\]

Substituting the formulas for \(d\) and \(n\), we get
\[
\begin{aligned}
h
={}&
[01]
+\frac15[02]
-\frac65[04]
+\frac45[13]
+\frac15[14] \\
&-\frac45[23]
+[24].
\end{aligned}
\]

Thus
\[
\langle h,[01]\rangle=1,
\langle h,[24]\rangle=1,
\langle h,[04]\rangle=-\frac65.
\]


We now compute the essential edges of the bar \(b\).
For a general representative
$
a d+b n\in \Rep(b),
 a\neq 0,
$
the edge coefficients are
\[
\begin{array}{c|c}
\text{edge} & \text{coefficient in } a d+b n\\
\hline
[01] & a \\ 

[24] & a \\ 

[04] & -a-b\\

[13] & a-b \\ 

[23] & -a+b  \\ 

[02] & b\\ 

[14] & b
\end{array}
\]

Since \(a\neq 0\), the coefficients of \([01]\) and \([24]\) are nonzero in
every representative. Therefore
\[
[01],[24]\in \Sigma(b).
\]

On the other hand, the coefficient of \([04]\) is
$
-a-b
$,
which can be made zero by choosing
$
b=-a
$.
Hence
\[
[04]\notin \Sigma(b).
\]

Thus
\[
\sigma=[01]\in \Sigma(b),
\tau=[04]\notin \Sigma(b).
\]

For the harmonic representative
$
h=d+\frac15 n,
$
we have
\[
|\langle h,\sigma\rangle|
=
|\langle h,[01]\rangle|
=
1,
\]
whereas
\[
|\langle h,\tau\rangle|
=
|\langle h,[04]\rangle|
=
\frac65.
\]

Therefore
\[
|\langle h,\tau\rangle|
>
|\langle h,\sigma\rangle|.
\]

This contradicts the coefficient-dominance conclusion
$
|\langle h,\tau\rangle|
<
|\langle h,\sigma\rangle|
$
with
$
\sigma\in\Sigma(b),
\tau\notin\Sigma(b).
$

\hide{
The bar
\[
b=(1,3;1)
\]
is simple but not generic, because another bar with the same birth time is
present:
\[
(1,2;1).
\]

For this non-generic simple bar, the harmonic representative is
\[
h
=
[01]
+\frac15[02]
-\frac65[04]
+\frac45[13]
+\frac15[14]
-\frac45[23]
+[24],
\]
and the essential edge
\[
\sigma=[01]
\]
has coefficient magnitude \(1\), while the non-essential edge
\[
\tau=[04]
\]
has coefficient magnitude \(6/5\). Thus
\[
|\langle h,\tau\rangle|>|\langle h,\sigma\rangle|.
\]

Therefore the genericity condition that no other bar has the same birth time
cannot be omitted from the coefficient-dominance theorem.
}

\section{Conclusion and future work}
\label{sec:conclusion}
In this paper we have proved that, for generic one-dimensional bars belonging to the barcode of a finite filtration, every essential edge has strictly larger coefficient, in absolute value, than every non-essential edge in the harmonic representative, and the absolute values of the coefficients of all essential edges are equal. We have also shown that the corresponding statement is not true in general for higher dimensional bars by giving a
sequence of counter-examples. However, our counter-examples are quite special. It would be interesting to investigate under what conditions on the filtration the coefficient dominance theorem holds for higher dimensional homology as well. Another open problem is to investigate the stability of the inequality
in Theorem~\ref{thm:main} under small perturbations of the filtration (along the same lines that stability has been investigated in the context of harmonic persistent homology \cite{Basu-Cox,gulen2025grassmannianpersistencediagramsspecial}).

The coefficient dominance of the essential simplices in harmonic representatives (in dimension one) indicates that the harmonic weights carry important information about the underlying filtration. It would be interesting to exploit this further in practical applications and empirically check the validity of this observation. One first step in this direction was taken in \cite{Gurnari-et-al}. 

\section*{Acknowledgements}
We acknowledge the use of the AI programs \href{https://chatgpt.com}{Chatgpt} and 
\href{https://claude.ai}{Claude}
for assistance with exploring counterexamples during the brainstorming phase, as well as for reviewing the manuscript and offering suggestions during the paper-writing phase. All mathematical proofs, derivations, and final verifications were performed and checked independently by the authors.

\bibliographystyle{amsplain}
\bibliography{bibliography}

@article {Basu-Cox,
    AUTHOR = {Basu, Saugata and Cox, Nathanael},
     TITLE = {Harmonic persistent homology},
   JOURNAL = {SIAM J. Appl. Algebra Geom.},
  FJOURNAL = {SIAM Journal on Applied Algebra and Geometry},
    VOLUME = {8},
      YEAR = {2024},
    NUMBER = {1},
     PAGES = {189--224},
      ISSN = {2470-6566},
   MRCLASS = {55N31},
  MRNUMBER = {4722356},
MRREVIEWER = {Facundo\ M\'emoli},
       DOI = {10.1137/22M1518761},
       URL = {https://doi.org/10.1137/22M1518761},
note={(an extended abstract appears in the
Proceedings of FOCS, 2021)},
}

@article{Eckmann,
	Author = {Beno Eckmann},
	Journal = {Commentarii Mathematici Helvetici},
	Number = {1},
	Pages = {240-255},
	Title = {Harmonische Funktionen und Randwertaufgaben in einem Komplex},
	Volume = {17},
	Year = {1944}}

@inproceedings{essential,
	Author = {Saugata Basu and Filippo Utro and Laxmi Parida},
	Booktitle = {18th International Workshop on Algorithms in Bioinformatics},
	Editor = {Laxmi Parida and Esko Ukkonen},
	Pages = {14:1-14:10},
	Publisher = {Schloss Dagstuhl},
	Series = {Leibniz International Proceedings in Informatics},
	Title = {Essential Simplices in Persistent Homology and Subtle Admixture Detection},
	Volume = {113},
	Year = {2018}}

@article{volume,
	Author = {Ippei Obayashi},
	Journal = {SIAM Journal on Applied Algebra and Geometry},
	Number = {4},
	Pages = {508-534},
	Title = {Volume-Optimal Cycle: TIght Representative Cycle of a Generator in Persistent Homology},
	Volume = {2},
	Year = {2018}}

@incollection {T,
    AUTHOR = {Thom, R.},
     TITLE = {Sur l'homologie des vari\'et\'es alg\'ebriques r\'eelles},
 BOOKTITLE = {Differential and Combinatorial Topology (A Symposium in Honor
              of Marston Morse)},
     PAGES = {255--265},
 PUBLISHER = {Princeton Univ. Press},
   ADDRESS = {Princeton, N.J.},
      YEAR = {1965},
   MRCLASS = {57.31 (57.50)},
  MRNUMBER = {0200942 (34 \#828)},
MRREVIEWER = {R. Bott},
}

@article {Lim2020,
    AUTHOR = {Lim, Lek-Heng},
     TITLE = {Hodge {L}aplacians on graphs},
   JOURNAL = {SIAM Rev.},
  FJOURNAL = {SIAM Review},
    VOLUME = {62},
      YEAR = {2020},
    NUMBER = {3},
     PAGES = {685--715},
      ISSN = {0036-1445},
   MRCLASS = {58A14 (05C50 20G10)},
  MRNUMBER = {4131346},
       DOI = {10.1137/18M1223101},
       URL = {https://doi-org.ezproxy.lib.purdue.edu/10.1137/18M1223101},
}

@article {Basu-Karisani,
    AUTHOR = {Basu, Saugata and Karisani, Negin},
     TITLE = {Persistent homology of semialgebraic sets},
   JOURNAL = {SIAM J. Appl. Algebra Geom.},
  FJOURNAL = {SIAM Journal on Applied Algebra and Geometry},
    VOLUME = {7},
      YEAR = {2023},
    NUMBER = {3},
     PAGES = {651--684},
      ISSN = {2470-6566},
   MRCLASS = {55N31 (14P10 14Q30 68U03)},
  MRNUMBER = {4646856},
       DOI = {10.1137/22M1494415},
       URL = {https://doi.org/10.1137/22M1494415},
}

@book {Edelsbrunner-Harer2010,
    AUTHOR = {Edelsbrunner, Herbert and Harer, John L.},
     TITLE = {Computational topology},
      NOTE = {An introduction},
 PUBLISHER = {American Mathematical Society},
   ADDRESS = {Providence, RI},
      YEAR = {2010},
     PAGES = {xii+241},
      ISBN = {978-0-8218-4925-5},
   MRCLASS = {00-02 (05C10 52-02 55-02 57-02 65D18 68U05)},
  MRNUMBER = {2572029 (2011e:00001)},
MRREVIEWER = {Andrzej Kozlowski},
}

@misc{talk:Lieutier,
author = {Lieutier, A.},
title = {Talk: Persistent Harmonic Forms},
note   = "URL: \url{https://project.inria.fr/gudhi/files/2014/10/Persistent-Harmonic-Forms.pdf}. 
            Last visited on 2021/05/28",
organization = {INRIA},
date = {Oct, 2014},
urldate = {}
}

@article{Chen-et-al-2021,
title = {Evolutionary de {R}ham-{H}odge method},
journal = {Discrete \& Continuous Dynamical Systems - B},
volume = {26},
number = {7},
pages = {3785-3821},
year = {2021},
author = {Jiahui Chen and  Rundong Zhao and Yiying Tong and Guo-Wei Wei},
}

@inproceedings{Dey-2010,
  author    = {Tamal K. Dey and
               Anil N. Hirani and
               Bala Krishnamoorthy},
  editor    = {Leonard J. Schulman},
  title     = {Optimal homologous cycles, total unimodularity, and linear programming},
  booktitle = {Proceedings of the 42nd {ACM} Symposium on Theory of Computing, {STOC}
               2010, Cambridge, Massachusetts, USA, 5-8 June 2010},
  pages     = {221--230},
  publisher = {{ACM}},
  year      = {2010},
  url       = {https://doi.org/10.1145/1806689.1806721},
  doi       = {10.1145/1806689.1806721},
  bibsource = {dblp computer science bibliography, https://dblp.org}
}

@article{lockwood2014, 
    title={Topological Features In Cancer Gene Expression Data}, DOI={10.1142/9789814644730_0012}, 
    journal={Biocomputing 2015}, 
    author={Lockwood, S. and Krishnamoorthy, B.}, 
    year={2014}
}

@article{Memoli-et-al,
	doi = {10.1137/21m1435471},
  
	url = {https://doi.org/10.1137%2F21m1435471},
  
	year = 2022,
	month = {jun},
  
	publisher = {Society for Industrial {\&} Applied Mathematics ({SIAM})},
  
	volume = {4},
  
	number = {2},
  
	pages = {858--884},
  
	author = {Facundo Memoli and Zhengchao Wan and Yusu Wang},
  
	title = {Persistent Laplacians: Properties, Algorithms and Implications},
  
	journal = {{SIAM} Journal on Mathematics of Data Science}
}

@misc{Ghrist-Henselman,
  doi = {10.48550/ARXIV.2112.04927},
  
  url = {https://arxiv.org/abs/2112.04927},
  
  author = {Ghrist, Robert and Henselman-Petrusek, Gregory},
  
  title = {Saecular persistence},
  
  publisher = {arXiv},
  
  year = {2021},
  
  copyright = {Creative Commons Attribution 4.0 International}
}

@article{Gurnari-et-al,
	author = {Gurnari, Davide and Guzm{\'a}n-S{\'a}enz, Aldo and Utro, Filippo and Bose, Aritra and Basu, Saugata and Parida, Laxmi},
	date = {2025/11/06},
	doi = {10.1038/s41598-025-12189-y},
	id = {Gurnari2025},
	isbn = {2045-2322},
	journal = {Scientific Reports},
	number = {1},
	pages = {38836},
	title = {Probing omics data via harmonic persistent homology},
	url = {https://doi.org/10.1038/s41598-025-12189-y},
	volume = {15},
	year = {2025}}

@incollection{hou2024,
  author       = {Tao Hou and
                  Salman Parsa and
                  Bei Wang},
  editor       = {Oswin Aichholzer and
                  Haitao Wang},
  title        = {Tracking the Persistence of Harmonic Chains: Barcode and Stability},
  booktitle    = {41st {I}nternational {S}ymposium on {C}omputational {G}eometry, {S}o{C}{G} 2025,
                  June 23-27, 2025, Kanazawa, Japan},
  series       = {LIPIcs},
  volume       = {332},
  pages        = {58:1--58:16},
  publisher    = {Schloss Dagstuhl - Leibniz-Zentrum f{\"{u}}r Informatik},
  year         = {2025},
  url          = {https://doi.org/10.4230/LIPIcs.SoCG.2025.58},
  doi          = {10.4230/LIPICS.SOCG.2025.58},
  bibsource    = {dblp computer science bibliography, https://dblp.org}
}

@misc{gulen2025grassmannianpersistencediagramsspecial,
      title={Grassmannian Persistence Diagrams: Special Properties in the 1-Parameter Setting}, 
      author={Aziz Burak Gülen and Facundo Mémoli and Zhengchao Wan},
      year={2025},
      eprint={2504.06077},
      archivePrefix={arXiv},
      primaryClass={math.CO},
      url={https://arxiv.org/abs/2504.06077}, 
}

@article{gyurik-et-al,
   title={Provable Quantum Speedups for Computing Persistence in Topological Data Analysis},
   volume={7},
   ISSN={2691-3399},
   url={http://dx.doi.org/10.1103/gvys-hl8h},
   DOI={10.1103/gvys-hl8h},
   number={2},
   journal={PRX Quantum},
   publisher={American Physical Society (APS)},
   author={Gyurik, Casper and Schmidhuber, Alexander and King, Robbie and Dunjko, Vedran and Hayakawa, Ryu},
   year={2026},
   month=June}

@book {Lyons-Peres-book,
    AUTHOR = {Lyons, Russell and Peres, Yuval},
     TITLE = {Probability on trees and networks},
    SERIES = {Cambridge Series in Statistical and Probabilistic Mathematics},
    VOLUME = {42},
 PUBLISHER = {Cambridge University Press, New York},
      YEAR = {2016},
     PAGES = {xv+699},
      ISBN = {978-1-107-16015-6},
   MRCLASS = {60C05 (05C05 05C81 28A80 60J10 60J80 60K35 82B41)},
  MRNUMBER = {3616205},
MRREVIEWER = {Laurent\ Miclo},
       DOI = {10.1017/9781316672815},
       URL = {https://doi.org/10.1017/9781316672815},
}

@book {Doyle-Snell-book,
    AUTHOR = {Doyle, Peter G. and Snell, J. Laurie},
     TITLE = {Random walks and electric networks},
    SERIES = {Carus Mathematical Monographs},
    VOLUME = {22},
 PUBLISHER = {Mathematical Association of America, Washington, DC},
      YEAR = {1984},
     PAGES = {xiv+159},
      ISBN = {0-88385-024-9},
   MRCLASS = {94C05 (60J15 94C15)},
  MRNUMBER = {920811},
MRREVIEWER = {H.\ Kesten},
}

@article{Hayakawa-et-al-2026,
  author       = {Ryu Hayakawa and
                  Kuo{-}Chin Chen and
                  Min{-}Hsiu Hsieh},
  title        = {Quantum Walks on Simplicial Complexes and Harmonic Homology: Application
                  to Topological Data Analysis with Superpolynomial Speedups},
  journal      = {Quantum},
  volume       = {10},
  pages        = {2138},
  year         = {2026},
  url          = {https://doi.org/10.22331/q-2026-06-15-2138},
  doi          = {10.22331/Q-2026-06-15-2138},
  bibsource    = {dblp computer science bibliography, https://dblp.org}
}

\end{document}